\documentclass[12pt]{amsart}
\usepackage{amssymb,stmaryrd,bbold}
\usepackage{fullpage}

\newtheorem{theorem}{Theorem}

\newtheorem{corollary}[theorem]{Corollary}

\newtheorem{lemma}[theorem]{Lemma}

\newtheorem{observation}[theorem]{Observation}

\newtheorem{definition}[theorem]{Definition}

\newcommand{\QED}{\end{proof}}

\newcommand{\possible}{\mathop{\raisebox{-1pt}{$\Diamond$}}}
\newcommand{\necessary}{\mathop{\raisebox{-1pt}{$\Box$}}}
\newcommand{\ZFC}{{\rm ZFC}}

\newcommand{\CH}{{\rm CH}}
\newcommand{\GCH}{{\rm GCH}}
\newcommand{\MA}{{\rm MA}}

\newcommand{\smalllt}{\mathrel{\mathchoice{\raise2pt\hbox{$\scriptstyle<$}}{\raise1pt\hbox{$\scriptstyle<$}}{\raise0pt\hbox{$\scriptscriptstyle<$}}{\scriptscriptstyle<}}}
\newcommand{\ltkappa}{{{\smalllt}\kappa}}
\newcommand{\ltomega}{{{\smalllt}\omega}}
\newcommand{\ltlambda}{{{\smalllt}\lambda}}
\newcommand{\lttheta}{{{\smalllt}\theta}}

\newcommand{\forces}{\Vdash}

\newcommand{\Cof}{\mathop{\rm Cof}}
\newcommand{\Add}{\mathop{\rm Add}}
\newcommand{\Coll}{\mathop{\rm Coll}}

\newcommand{\axiomf}[1]{{\rm #1}}
\newcommand{\theoryf}[1]{\hbox{$\mathsf{#1}$}}

\newcommand{\of}{\subseteq}
\newcommand{\ofneq}{\subsetneq}

\newcommand{\B}{{\mathbb B}}
\newcommand{\C}{{\mathbb C}}
\newcommand{\bbS}{{\mathbb S}}
\renewcommand{\P}{{\mathbb P}}
\newcommand{\Q}{{\mathbb Q}}
\newcommand{\R}{{\mathbb R}}
\newcommand{\T}{{\mathbb T}}
\newcommand{\Rdot}{{\dot\R}}

\newcommand{\plus}{{+}}

\newcommand{\set}[1]{\{\,{#1}\,\}}
\newcommand{\singleton}[1]{\left\{{#1}\right\}}
\newcommand{\st}{\mid}

\newcommand{\iso}{\cong}
\newcommand{\cross}{\times}
\newcommand{\satisfies}{\models}

\newcommand{\ORD}{\mathop{{\rm ORD}}}
\newcommand{\elesub}{\prec}
\newcommand{\union}{\cup}

\newcommand{\Union}{\bigcup}

\newcommand{\intersect}{\cap}

\newcommand{\df}{\it} 

\def\<#1>{\langle#1\rangle}

\newcommand{\onex}{\mathbb{1}}
\newcommand{\zerox}{\mathbb{0}}

\newcommand{\LL}{\mathrm{L}}
\renewcommand{\implies}{\rightarrow}
\renewcommand{\iff}{\leftrightarrow}
\newcommand{\Lmodal}{\mathcal{L}_{\tiny\necessary}}
\newcommand{\Lst}{\mathcal{L}_{\tiny\in}}
\begin{document}

\title[Connections between a forcing class and its logic]{Structural connections between a forcing class and its modal logic}

\author{Joel David Hamkins}

\address[J.D.H.]{Mathematics Program, The Graduate Center of The City University of New York, 365 Fifth Avenue, New York, NY 10016, U.S.A. \& Department of Mathematics, The College of Staten Island of CUNY, Staten Island, NY 10314, U.S.A.}

\email{jhamkins@gc.cuny.edu, http://jdh.hamkins.org}

\thanks{The research of the first author has been supported in part by NSF grant DMS-0800762, PSC-CUNY grant 64732-00-42 and Simons Foundation grant 209252.
The third author would like to thank the CUNY Graduate Center in New York for their hospitality during his sabbatical in the fall of 2009,
and he acknowledges the generous support of the \textsl{Vlaams Academisch Centrum} at the Royal Flemish Academy in Brussels in the form of
a  fellowship in 2010/11. In addition, the first and third authors both acknowledge the generous support provided to them as visiting
fellows in Spring 2012 at the Isaac Newton Institute for Mathematical Sciences in Cambridge, U.K., where this article was completed. The authors would like to thank Nick Bezhanishvili for input on \S\,\ref{Section.ModalLogicBackground}.}

\author{George Leibman}

\address[G.L.]{Department of Mathematics and Computer Science, Bronx Community College, The City University of New York, 2155 University Avenue, Bronx, NY 10453}

\email{gleibman@acedsl.com}

\author{Benedikt L\"owe}
\address[B.L.]{Institute for Logic, Language and Computation, Universiteit van Amsterdam,
Postbus 94242, 1090 GE Amsterdam, The Netherlands \& Fachbereich Mathematik, Universit\"at Hamburg, Bundesstra{\ss}e 55, 20146 Hamburg,
Germany \& Isaac Newton Institute for Mathematical Sciences, 20, Clarkson Road, Cambridge CB3 0EH, England \&
Corpus Christi College, University of Cambridge, Trumpington Street,
Cambridge CB2 1RH, England}

\email{bloewe@science.uva.nl}


\begin{abstract}
Every definable forcing class $\Gamma$ gives rise to a corresponding forcing modality, for which $\necessary_\Gamma\varphi$ means that
$\varphi$ is true in all $\Gamma$ extensions, and the valid principles of $\Gamma$ forcing are the modal assertions that are valid for this
forcing interpretation. For example, \cite{HamkinsLoewe2008:TheModalLogicOfForcing} shows that if\/ \ZFC\ is consistent, then the
\ZFC-provably valid principles of the class of all forcing are precisely the assertions of the modal theory \theoryf{S4.2}. In this
article, we prove similarly that the provably valid principles of collapse forcing, Cohen forcing and other classes are in each case
exactly \theoryf{S4.3}; the provably valid principles of c.c.c.~forcing, proper forcing, and others are each contained within
\theoryf{S4.3} and do not contain \theoryf{S4.2}; the provably valid principles of countably closed forcing, \CH-preserving forcing and
others are each exactly \theoryf{S4.2}; and the provably valid principles of $\omega_1$-preserving forcing are contained within
\theoryf{S4.tBA}. All these results arise from general structural connections we have identified between a forcing class and the modal
logic of forcing to which it gives rise.
\end{abstract}

\maketitle


\section{Introduction}

In \cite{HamkinsLoewe2008:TheModalLogicOfForcing}, we considered the \emph{modal logic of forcing}, which arises when one considers a model
of set theory in the context of all its forcing extensions, interpreting $\necessary$ as ``in all forcing extensions'' and $\possible$ as
``in some forcing extension''. This modal language allows one easily to express sweeping general principles concerning forcing absoluteness
and the effect of forcing on truth in set theory, such as the assertion $\possible\necessary\varphi\implies\necessary\possible\varphi$,
expressing that every possibly necessary statement is necessarily possible, which the reader may verify is valid for the forcing
interpretation, or the assertion $\possible\necessary\varphi\implies\varphi$, that every possibly necessary statement is true, which is an
equivalent formulation of the maximality principle \cite{Hamkins2003:MaximalityPrinciple,StaviVaananen:ReflectionPrinciples}, a forcing
axiom independent of but equiconsistent with \ZFC. It was known from \cite{Hamkins2003:MaximalityPrinciple} that the valid principles of
forcing include all the assertions of the modal theory \theoryf{S4.2}, and the main theorem of
\cite{HamkinsLoewe2008:TheModalLogicOfForcing} established that if $\ZFC$ is consistent, then in fact the \ZFC-provably valid principles of
forcing are exactly the assertions of \theoryf{S4.2}.

In this article, we consider more generally the modal logic of various particular kinds of forcing, such as c.c.c.~forcing or $\omega_1$-preserving forcing, by relativizing the modal operators to such a class. In
\cite{HamkinsLoewe2008:TheModalLogicOfForcing}, a number of connections between various modal axioms and corresponding operations
on the class of forcing notions had surfaced: for instance, the validity of the axiom $\axiomf{.2}$ mentioned above is linked to the operation of finite products of forcing notions \cite[p.\,1794]{HamkinsLoewe2008:TheModalLogicOfForcing}. In this
article, we shall uncover further such structural connections between a class of forcing notions and the modal logic of forcing to which it
gives rise, and use these connections to settle the exact modal logic of several natural forcing classes, while providing new bounds on
several others. We have aimed to provide general tools, including the control statements of \S\ref{Section.ControlStatements} and their
connection to the forcing validities, which seem promising to enable set theorists to undertake an investigation of the modal logic of
additional forcing classes by means of a forcing-only analysis (requiring no substantial modal logic), by discovering which kinds of
control statements their forcing class supports.

For example, using these methods we prove in this article that
the provably valid principles of collapse forcing, Cohen forcing and other
classes are in each case exactly $\theoryf{S4.3}$ (theorems \ref{Theorem.CollValid=S4.3} and \ref{Theorem.AddValid=S4.3}); the provably valid principles of c.c.c.~forcing, proper
forcing, and others are each contained within $\theoryf{S4.3}$ and do not contain $\theoryf{S4.2}$ (corollary \ref{Corollary.VariousClassesInS4.3}); the provably valid principles of countably closed forcing, $\CH$-preserving forcing and others are each exactly $\theoryf{S4.2}$ (theorems \ref{Theorem.CountablyClosed=S4.2} and \ref{thm:chpreserving}); and the provably valid principles of $\omega_1$-preserving forcing are contained within
$\theoryf{S4.tBA}$ (theorem \ref{Theorem.Omega1PrservingS4tBA}). 

The forcing interpretation of modal logic was introduced in \cite{Hamkins2003:MaximalityPrinciple}, with a deeper more explicit
investigation in \cite{HamkinsLoewe2008:TheModalLogicOfForcing}. Various other aspects of the modal logic of forcing are considered in
\cite{Leibman2004:Dissertation,HamkinsWoodin2005:NMPccc,Fuchs2008:ClosedMaximalityPrinciples,Fuchs2009:CombinedMaximalityPrinciplesUpToLargeCardinals,Leibman2010:TheConsistencyStrengthOfMPccc(R),Rittberg2010:TheModalLogicOfForcing,FriedmanFuchinoSakai:OnTheSetGenericMultiverse,EsakiaLoewe}. We consider this article a natural
successor to \cite{HamkinsLoewe2008:TheModalLogicOfForcing}.

\section{Modal logic background}
\label{Section.ModalLogicBackground}

The modal logic of forcing involves an interplay between two formal languages, which we sharply distinguish, namely (i) the language of
propositional modal logic $\Lmodal$, which has propositional variables, logical connectives and the modal operators $\necessary$ and $\possible$ (where $\possible\varphi$ is defined as $\neg\necessary\neg\varphi$; and (ii) the usual
first-order language of set theory $\Lst$, which has variables, quantifiers, logical connectives and the relations $=$ and $\in$. One may regard a modal
assertion $\varphi(p_1,\ldots,p_n)$ as a template for the scheme of set-theoretic assertions $\varphi(\psi_1,\ldots,\psi_n)$ which arise by
the substitution of set-theoretic assertions $\psi_i$ for the propositional variables $p_i$ and the forcing interpretation of the modal
operators. The nature of this translation will be further investigated in \S\ref{Section.ForcingWithGamma}. Meanwhile, in this section, let
us introduce the modal theories that will arise in that analysis. Using the following axioms (with their established nomenclature)
$$\begin{array}{rl}
 \axiomf{K} &\ \necessary(\varphi\implies \psi)\implies(\necessary\varphi\implies\necessary\psi)\hskip.7in\\
 \axiomf{T} &\ \necessary\varphi\implies\varphi\\
 \axiomf{4} &\ \necessary\varphi\implies \necessary\necessary\varphi\\
 \axiomf{.2} &\ \possible\necessary\varphi\implies\necessary\possible\varphi\\
 \axiomf{.3} &\ (\possible\varphi\wedge\possible\psi)\implies\possible[(\varphi\wedge\possible\psi)\vee(\psi\wedge\possible\varphi)]\\
 \axiomf{5} &\ \possible\necessary\varphi\implies\varphi,\\
\end{array}$$
we define the desired modal theories, where in each case we close the axioms under modus ponens and necessitation:
$$\begin{array}{rcl}
 \theoryf{S4} &=& \axiomf{K}+\axiomf{T}+\axiomf{4}\hskip2.3in\\
 \theoryf{S4.2} &=& \axiomf{K}+\axiomf{T}+\axiomf{4}+\axiomf{.2}\\
 \theoryf{S4.3} &=& \axiomf{K}+\axiomf{T}+\axiomf{4}+\axiomf{.3}\\
 \theoryf{S5} &=& \axiomf{K}+\axiomf{T}+\axiomf{4}+\axiomf{5}.\\
\end{array}$$
It is not difficult to see that $\theoryf{S4}\vdash \axiomf{5}\to\axiomf{.3}$ and $\theoryf{S4}\vdash\axiomf{.3}\to\axiomf{.2}$, and
consequently $\theoryf{S4}\of\theoryf{S4.2}\of\theoryf{S4.3}\of\theoryf{S5}$.

The Kripke model concept provides a robust semantics for modal logic: a {\df Kripke model} is a collection $M$ of possible worlds, each
providing truth values for the propositional variables, together with a relation on the worlds
called the accessibility relation. The {\df frame} of such a model is this accessibility relation, disregarding the truth assignments of the worlds. Modal truth is defined by induction in the natural way, so that $\necessary\varphi$ is true at a world in $M$, if $\varphi$ is true in all accessible worlds, and similarly $\possible\varphi$ is true at a world if $\varphi$ is true at some accessible world. We write $\llbracket v\rrbracket_M$ for the set of $\Lmodal$-formulas true at $v$ in $M$ and note that this is a set closed under modus ponens.

If $F$ is a frame, a modal assertion is {\df valid for $F$} if it is true at all worlds of all Kripke models having $F$ as a frame.  If $\mathcal{C}$ is a class of frames, a modal theory is {\df sound with respect to $\mathcal{C}$} if every assertion in the theory is valid for every frame in $\mathcal{C}$.  A modal theory is {\df complete with respect to $\mathcal{C}$} if every assertion valid for every frame in $\mathcal{C}$ is in the theory.  Finally, a modal theory is {\df characterized by $\mathcal{C}$} (equivalently, {\df $\mathcal{C}$ characterizes} the modal theory) if it is both sound and complete with respect to $\mathcal{C}$ \cite[p.\ 40]{HughesCresswell1996:ANewIntroductionToModalLogic}.


It turns out that all the modal theories we mentioned are characterized by certain classes of finite Kripke frames. For the rest of this section, we mention the completeness results of which we shall subsequently make use. Numerous further completeness results can be found in \cite{ChagrovZakharyaschev1997:ModalLogic}.

A {\df pre-order} is a reflexive transitive relation. Every pre-order $\leq$ admits a quotient by the equivalence $x\equiv y\iff x\leq
y\leq x$, and this quotient will be a partial order. Thus, a pre-order is obtained from a partial order by replacing each node with a
cluster of equivalent nodes. A {\df linear pre-order} is a pre-order in which any two nodes are comparable.

\begin{theorem}\label{Theorem.PreLinearOrdersCompleteForS4.3}
The modal logic \theoryf{S4.3} is characterized by the class of finite linear pre-order frames.
That is, a modal assertion is derivable in \theoryf{S4.3} if
and only if it holds in all Kripke models having a finite linear pre-ordered frame.
\end{theorem}

\begin{proof} Cf.\ \cite[corollary\,5.18 \& theorem\,5.33]{ChagrovZakharyaschev1997:ModalLogic} or \cite[exercise\,4.33, theorem\,4.96, \& lemma\,6.40]{BlackburnDeRijkeVenema2001:ModalLogic}.
\end{proof}

Every finite Boolean algebra partial order is isomorphic to the set of subsets of a finite set, ordered by inclusion. Of course, this
algebra also has the structure of meets (intersection), joins (unions) and negation (complements), with a smallest element $\zerox$ and a
largest element $\onex$. A {\df pre-Boolean algebra} is a partial pre-order $\langle\B,\leq\rangle$, such that the quotient by the relation $x\equiv
y\iff x\leq y\leq x$ is a Boolean algebra. Thus, every pre-Boolean algebra is obtained from a Boolean algebra by replacing each node in the
Boolean algebra by a cluster of equivalent nodes.

\begin{theorem}[{\cite[theorem\,11]{HamkinsLoewe2008:TheModalLogicOfForcing}}]
\label{Theorem.CanopiedTreesCompleteForS4.2}
The modal logic \theoryf{S4.2} is characterized by
the class of finite pre-Boolean algebras. That is, a modal
assertion is derivable in \theoryf{S4.2} if and only if it holds in all Kripke models having a finite pre-Boolean algebra frame.
\end{theorem}

\begin{theorem}\label{Theorem.EquivalenceRelationsCompleteForS5}
The modal logic \theoryf{S5} is characterized by the class of finite equivalence relations
with one equivalence class (a single cluster).
\end{theorem}

\begin{proof} Cf.\ \cite[corollaries\,5.19 \& 5.29]{ChagrovZakharyaschev1997:ModalLogic} or \cite[theorems\,4.29, 4.96 \& exercise\,6.6.4]{BlackburnDeRijkeVenema2001:ModalLogic}.\end{proof}

We have observed another modal theory to arise several times in the modal logic of various forcing classes, and so let us now introduce it. This is the logic of finite topless pre-Boolean algebras. A {\df topless} Boolean algebra is obtained from a Boolean algebra by omitting the top element $\onex$. Thus, a finite topless Boolean algebra is isomorphic to the collection of strictly proper subsets of a given finite
set. A topless pre-Boolean algebra is a partial pre-order (transitive and reflexive), such that the natural quotient by the relation $x\equiv y\iff x\leq y\leq x$ is a topless Boolean algebra. We define the modal theory \theoryf{S4.tBA}, or {\df topless pre-Boolean algebra logic}, to be the collection of all modal assertions that are true in all Kripke models whose frame is a finite topless pre-Boolean algebra. Thus, by definition, this logic is complete with respect to the class of finite topless pre-Boolean algebras. This logic is the \emph{smallest modal companion} of a well-known intermediate logic called \emph{Medvedev's Logic} $\theoryf{ML}$.\footnote{An intermediate logic is a propositional logic between intuitionistic and classical logic. A modal logic $\Lambda$ is a {\df modal companion} of an intermediate logic $\Lambda'$ if $\Lambda$ consists of the G\"odel translations of formulas in $\Lambda'$. If $\Lambda$ is characterized by a class $\mathcal{C}$ in the above sense, and $\Lambda'$ is characterized in the sense of Kripke semantics for intermediate logics by the same class $\mathcal{C}$, then $\Lambda$ is the smallest modal companion of $\Lambda'$. Medvedev's Logic $\theoryf{ML}$ is characterized by the class of finite topless pre-Boolean algebras and known not to be finitely axiomatizable \cite{Gabbay1970,Maksimovaetal1979}.}

\begin{observation}
 \ \theoryf{S4.tBA} is properly contained within
\theoryf{S4.2}.
\end{observation}

\begin{proof}
First, we argue that \theoryf{S4.tBA} is contained in \theoryf{S4.2} by arguing that every Kripke model $M$ with a finite pre-Boolean algebra frame admits a model bisimilarity with a Kripke model $M^+$ having a finite topless pre-Boolean algebra frame. Consider first the case where the frame underlying $M$ is actually a Boolean algebra. We construct $M^+$  as follows: take the Boolean algebra that is the frame of $M$, add a new atom and consider the resulting generated Boolean algebra; then remove the top to create a topless Boolean algebra. In the resulting frame, keep the original worlds for the nodes that came from $M$ and place the top world of $M$ at each of the newly created nodes of the frame of $M^+$.
It is easy to check that $M$ and $M^+$ have a model bisimulation associating each of the new worlds to the top world of $M$, and the other worlds to themselves.
It follows that
the truths of $M$ and $M^+$ at their initial worlds are the same. Similarly, for the general case where the frame of $M$ is a finite
pre-Boolean algebra rather than a Boolean algebra, then there is a cluster of mutually accessible nodes at the top of $M$, and it is this
entire cluster that we duplicate in $M^+$ at each of the new positions in the topless Boolean algebra. Again there is a model bisimulation
of $M$ and $M^+$ associating each of these newly created worlds with their corresponding duplicate in the top cluster of $M$ and the other
worlds to themselves. Thus, again the truths of $M$ and $M^+$ at their initial worlds are the same. Finally, since any statement outside
\theoryf{S4.2} must fail in such a Kripke model $M$ with a finite pre-Boolean algebra frame, it follows that the statement also fails in
the corresponding Kripke model $M^+$ with a finite topless pre-Boolean algebra frame, and so it is also outside \theoryf{S4.tBA}, as
desired. So \theoryf{S4.tBA} is included in \theoryf{S4.2}.

Second, we argue that this inclusion is strict. As the topless Boolean algebras are not directed, it is easy to construct violations of
\axiomf{.2} by making a statement true on one co-atom and false on another. So \axiomf{.2} is not in \theoryf{S4.tBA}, and consequently
$\theoryf{S4.tBA}\ofneq\theoryf{S4.2}$, as desired.
\end{proof}

Since \theoryf{S4} is valid in any topless Boolean algebra, it follows that $\theoryf{S4}\of\theoryf{S4.tBA}$. This conclusion is strict in
light of the fact that all topless Boolean algebras satisfy the principle that whenever three mutually incompatible assertions are possibly
necessary, then it is possible to exclude one of them without yet deciding between the other two. This is expressible in the modal
language as
$$\possible\necessary\varphi_1\wedge\possible\necessary\varphi_2\wedge\possible\necessary\varphi_3\wedge\neg\possible[(\varphi_1\wedge\varphi_2)\vee
(\varphi_1\wedge\varphi_3)\vee(\varphi_2\wedge\varphi_3)]$$
$$\implies\possible[\possible\necessary\varphi_1\wedge\possible\necessary\varphi_2\wedge\neg\possible\necessary\varphi_3],$$
and there are similar assertions for four possibilities and five and so on. The assertion above is valid for \theoryf{S4.tBA} but not for \theoryf{S4} (consider a
frame that is a tree with one root and three leaves), and thus we may
summarize the situation as
$$\theoryf{S4}\quad\ofneq\quad\theoryf{S4.tBA}\quad\ofneq\quad\theoryf{S4.2}\quad\ofneq\quad\theoryf{S4.3}\quad\ofneq\quad\theoryf{S5}.$$

We close this section with two simple modal logic observations that will be of use in \S\S\,\ref{ssec:preserving} and \ref{Section.CCCforcing}.

\begin{observation}\label{obs:nec.2}
The modal logic
\theoryf{S4} proves $\axiomf{.2}\implies\necessary\axiomf{.2}$.\end{observation}

\begin{proof} The following simple chain of transformations proves the claim, using the fact that $\possible\possible \varphi$ implies $\possible \varphi$ in \theoryf{S4}:
\begin{align*}
\neg\necessary\axiomf{.2} = \neg\necessary(\possible\necessary \varphi\to \necessary\possible \varphi) & \Rightarrow
  \possible(\possible\necessary\varphi \wedge \neg \necessary\possible\varphi)\\
  &\Rightarrow \possible\possible\necessary\varphi \wedge \possible\neg\necessary\possible\varphi\\
  &\Rightarrow \possible\possible\necessary\varphi \wedge \possible\possible\necessary\neg\varphi\\
  &\Rightarrow \possible\necessary \varphi \wedge \possible\necessary\neg\varphi\\
  &\Rightarrow \possible\necessary\varphi \wedge \neg\necessary\possible\varphi\\
  &\Rightarrow \neg(\possible\necessary\varphi \to \necessary\possible\varphi) = \neg\axiomf{.2}
\end{align*}
\end{proof}

\begin{observation}\label{obs:necessitation}
Let $\Lambda$ and $\mathrm{A}$ be sets of $\Lmodal$-formulas, and let $\Lambda^*$ be
the closure of $\Lambda \cup \mathrm{A}$ under modus ponens and necessitation. Let $M$ be a
transitive Kripke model and $v \in M$. Assume that
(1) for each $w \in M$, we have $\Lambda \subseteq \llbracket w\rrbracket_M$, and
(2) for each $\alpha \in \mathrm{A}$, we have $\Box \alpha \in \llbracket v\rrbracket_M$.
Then for every $w$ accessible from $v$, we have $\Lambda^* \subseteq \llbracket w\rrbracket_M$.
\end{observation}

\begin{proof}
The set $\Lambda^*$ can be written as $\bigcup_{n\in\omega} \Lambda_n$ where
$\Lambda_0 := \Lambda \cup \mathrm{A}$, $\Lambda_{2n+1}$ is the closure of $\Lambda_{2n}$ under
modus ponens, and $\Lambda_{2n+2} := \Lambda_{2n+1} \cup \set{\necessary\varphi \st \varphi \in
\Lambda_{2n+1}}$.

We'll prove by induction that $\Lambda^* \subseteq \llbracket w\rrbracket_M$ for all $w$ accessible
from $v$. The base case follows from assumptions (1) and (2). Since each $\llbracket w\rrbracket_M$
is closed under modus ponens, the step $2n \mapsto 2n+1$ is trivial. For the
step $2n+1 \mapsto 2n+2$, fix $w$ accessible from $v$ and $\varphi$ in $\Lambda_{2n+1}$, and
show that $\Box \varphi \in \llbracket w\rrbracket_M$. Let $w'$ be accessible from $w$; by transitivity, we
have that $w'$ is accessible from $v$, and thus, $\varphi \in \llbracket w'\rrbracket$. Since $w'$ was
arbitrary, $\Box\varphi \in \llbracket w\rrbracket$.
\end{proof}

\section{Forcing with a forcing class $\Gamma$}
\label{Section.ForcingWithGamma}

Suppose that $\Gamma$ is a fixed definable class of forcing notions, with a fixed definition. We freely interpret this definition in any
model of set theory, thereby reading $\Gamma$ {\it de dicto} rather than {\it de re}. We say that $\Gamma$ is a {\df forcing class} if in
every such model, this definition defines a class of forcing notions, closed under the equivalence of forcing. A forcing class $\Gamma$ is
{\df reflexive} if in every model of set theory, $\Gamma$ contains the trivial forcing poset. In this case, every model of set theory is a
$\Gamma$ forcing extension of itself. The class $\Gamma$ is {\df transitive} if it closed under finite iterations, in the sense that if
$\Q\in\Gamma$ and $\dot\R\in\Gamma^{V^\Q}$, then $\Q*\dot\R\in\Gamma$. Thus, any $\Gamma$ forcing extension of a $\Gamma$ forcing extension
is a $\Gamma$ forcing extension. The class $\Gamma$ is {\df closed under product forcing} if, necessarily, whenever $\Q$ and $\R$ are in
$\Gamma$, then so is $\Q\cross\R$. Related to this, $\Gamma$ is {\df persistent} if, necessarily, members of $\Gamma$ are $\Gamma$
necessarily in $\Gamma$; that is, if $\P,\Q\in\Gamma$ implies $\P\in\Gamma^{V^\Q}$ in all models. These properties are related in that
every transitive persistent class is closed under products, since $\Q\cross\R$ is forcing equivalent to $\Q*\check\R$. The class $\Gamma$
is {\df directed} if whenever $\P,\Q\in\Gamma$, then there is $\R\in\Gamma$, such that both $\P$ and $\Q$ are factors of $\R$ by further
$\Gamma$ forcing, that is, if $\R$ is forcing equivalent to $\P*\dot{\mathbb S}$ for some $\dot{\mathbb S}\in\Gamma^{V^\P}$ and also
equivalent to $\Q*\dot\T$ for some $\dot\T\in\Gamma^{V^\Q}$. This is stronger than asserting merely that any two members of $\Gamma$ are
absorbed by some other forcing in $\Gamma$, since we require that the quotient forcing is also in $\Gamma$. Note that if $\Gamma$ is
transitive and persistent, then product forcing shows that $\Gamma$ is directed. The class $\Gamma$ has the {\df linearity property} if for
any two forcing notions $\P,\Q$, then one of them is forcing equivalent to the other one followed by additional $\Gamma$ forcing; that is,
either $\P$ is forcing equivalent to $\Q*\Rdot$ for some $\Rdot\in\Gamma^{V^\Q}$ or $\Q$ is forcing equivalent to $\P*\Rdot$ for some
$\Rdot\in\Gamma^{V^\Q}$. Combining these notions, we define that $\Gamma$ is a {\df linear forcing class} if $\Gamma$ is reflexive,
transitive and has the linearity property. Any linear forcing class is clearly directed. Similar related definitions can be found in
\cite{Leibman2004:Dissertation}.

Any forcing class $\Gamma$ leads to the corresponding $\Gamma$ forcing modalities. Namely, a set-theoretic sentence $\psi$ is {\df
$\Gamma$-forceable} or {\df $\Gamma$-possible}, written $\possible_\Gamma\psi$, if $\psi$ holds in a forcing extension by some forcing
notion in $\Gamma$, and $\psi$ is {\df $\Gamma$-necessary}, written $\necessary_\Gamma\psi$, if $\psi$ holds in all forcing extensions by
forcing notions in $\Gamma$. These modal operators are easily seen to obey certain modal validities, such as the following, for any class
$\Gamma$.
$$\begin{array}{rl}
 \axiomf{K} & \necessary_\Gamma(\varphi\implies \psi)\implies(\necessary_\Gamma\varphi\implies\necessary_\Gamma\psi), \quad \text{and}\\
 \axiomf{Dual} & \possible_\Gamma \varphi\iff \neg\necessary_\Gamma\neg\varphi.\\
\end{array}$$
The validity of other statements will depend on the class $\Gamma$. A modal assertion $\varphi(p_0,\ldots,p_n)$ is a {\df valid principle
of\/ $\Gamma$ forcing} if all substitution instances $\varphi(\psi_0,\ldots,\psi_n)$ hold under the $\Gamma$-forcing interpretation, where
we substitute arbitrary set-theoretic assertions $\psi_i$ for the propositional variables $p_i$ and interpret $\necessary$ as
$\necessary_\Gamma$. More formally, for any forcing class $\Gamma$, every assignment $p_i\mapsto\psi_i$ of the propositional variables
$p_i$ to set-theoretical assertions $\psi_i$ extends recursively to a {\df $\Gamma$ forcing translation}, a function $H:\Lmodal\to\Lst$
satisfying $H(p_i)=\psi_i$ and the obvious recursive rules for conjunction $H(\varphi\wedge\eta)=H(\varphi)\wedge H(\eta)$, negation
$H(\neg\varphi)=\neg H(\varphi)$ and modality $H(\necessary\varphi)=\necessary_\Gamma H(\varphi)$, where in this last case we mean the
assertion in the language of set theory asserting that $H(\varphi)$ has Boolean value one for any forcing notion satisfying the definition
of $\Gamma$. In this terminology, the {\df modal logic of $\Gamma$ forcing} over a model of set theory $W$ is the set $$\set{\varphi\in
\Lmodal\st W\satisfies H(\varphi)\text{ for all }\Gamma\text{ forcing translations }H}.$$ Note that the interpretation of iterated modal
expressions, such as $\possible_\Gamma\necessary_\Gamma\psi$, leads to the situation where $\Gamma$ is reinterpreted in the appropriate
forcing extensions by forcing in $\Gamma$. For example, the c.c.c.~forcing notions of a c.c.c.~forcing extension do not necessarily include
all the c.c.c.~forcing notions of the original ground model. To improve readability, we will sometimes omit the subscripts from
$\possible_\Gamma$ and $\necessary_\Gamma$ when the class $\Gamma$ is clear.

\begin{theorem}[\cite{Leibman2004:Dissertation}]\ \label{Theorem.EasyValidities}
\begin{enumerate}
 \item\ \theoryf{S4} is valid for any reflexive transitive forcing class.
 \item\ \theoryf{S4.2} is valid for any reflexive transitive directed forcing class.
 \item\ \theoryf{S4.3} is valid for any linear forcing class.
\end{enumerate}
\end{theorem}

\begin{proof}
The reader is encouraged to work through the details as an exercise in understanding the forcing modalities. Axiom $\axiomf{K}$ is valid
for any forcing class. Axiom $\axiomf{T}$ is valid for any reflexive forcing class. Axiom $\axiomf{4}$ is valid for any transitive forcing
class. Axiom $\axiomf{.2}$ is valid for any directed forcing class. And axiom $\axiomf{.3}$ is valid for any forcing class with the
linearity property. We note that one subtlety of the argument is that it is not sufficient merely to check the validity of the axioms
$\axiomf{K}$, $\axiomf{T}$, $\axiomf{4}$ and $\axiomf{.3}$ alone, since the theory \theoryf{S4.3}, for example, is defined to be the
closure of these axioms under modus ponens {\it and necessitation}. Thus, one needs to know that the axioms remain valid in any $\Gamma$
extension. In our case here, the definition of what it means to be a linear class requires that the desired property holds in all models,
which is more than sufficient.
\end{proof}

The construction used in the next lemma lies at the heart of our method, serving as the principal technique connecting modal truth in
Kripke models with set-theoretic truth in models of set theory.  It is a general version of
\cite[lemma\,7.1]{HamkinsLoewe2008:TheModalLogicOfForcing}, employing the notion of a $\Gamma$-labeling, which we introduce here.

\begin{definition}\label{Definition.GammaLabeling}\rm
A {\df $\Gamma$-labeling} of a frame $F$ for a model of set theory $W$ is an assignment to each node $w$ in $F$ an assertion $\Phi_w$ in
the language of set theory, such that
 \begin{enumerate}
  \item The statements $\Phi_w$ form a mutually exclusive partition of truth in the $\Gamma$ forcing extensions of $W$, meaning that
      every such extension $W[G]$ satisfies exactly one $\Phi_w$.
  \item Any $\Gamma$ forcing extension $W[G]$ in which $\Phi_w$ is true satisfies $\possible\Phi_u$ if and only if $w\leq_F u$.
  \item $W\satisfies\Phi_{w_0}$, where $w_0$ is a given initial element of $F$.
 \end{enumerate}
\end{definition}
The formal assertion of these properties is called the Jankov-Fine formula for $F$ (cf. \cite[\S\,7]{HamkinsLoewe2008:TheModalLogicOfForcing}).

\begin{lemma}\label{Lemma.Labeling}
Suppose that $w\mapsto\Phi_w$ is a $\Gamma$-labeling of a finite frame $F$ for a model of set theory $W$ and that $w_0$ is an initial world
of $F$. Then for any Kripke model $M$ having frame $F$, there is an assignment of the propositional variables to set-theoretic assertions
$p\mapsto\psi_p$ such that for any modal assertion $\varphi(p_0,\ldots,p_k)$,
$$(M,w_0)\satisfies\varphi(p_0,\ldots,p_k)\quad\text{ iff }\quad W\satisfies\varphi(\psi_{p_0},\ldots,\psi_{p_k}).$$
In particular, any modal assertion $\varphi$ that fails at $w_0$ in $M$ also fails in $W$ under the $\Gamma$ forcing interpretation.
Consequently, the modal logic of $\Gamma$ forcing over $W$ is contained in the modal logic of assertions valid in $F$.
\end{lemma}

\begin{proof}
Suppose that $w\mapsto\Phi_w$ is a $\Gamma$-labeling of $F$ for $W$, and suppose that $M$ is a Kripke model with frame $F$. Thus, we may
view each $w\in F$ as a propositional world in $M$. For each propositional variable $p$, let $\psi_p=\bigvee\set{\Phi_w\st (M,w)\satisfies
p}$, the join of the set-theoretic statements $\Phi_w$ associated with each world $w$ where $p$ is true in $M$. We shall prove the lemma by
proving the more uniform claim that whenever $W[G]$ is a $\Gamma$ forcing extension of $W$ and $W[G]\satisfies\Phi_w$, then
 $$(M,w)\satisfies\varphi(p_0,\ldots,p_k)\quad\text{ iff }\quad
 W[G]\satisfies\varphi(\psi_{p_0},\ldots,\psi_{p_k}).$$
We prove this for all such $W[G]$ simultaneously by induction on the complexity of $\varphi$. The atomic case follows immediately from the
definition of $\psi_p$. Boolean combinations go through easily. Consider now the modal operator case. If
$W[G]\satisfies\possible\varphi(\psi_{p_0},\ldots,\psi_{p_k})$, then there is a further $\Gamma$ extension $W[G][H]$ satisfying
$\varphi(\psi_{p_0},\ldots,\psi_{p_k})$. This extension $W[G][H]$ must satisfy some $\Phi_u$, and consequently by induction we know
$(M,u)\satisfies\varphi(p_0,\ldots,p_k)$. Since $\Phi_u$ was $\Gamma$ forceable over $W[G]$, where $\Phi_w$ was true, it follows by the
labeling properties that $w\leq_F u$. Thus, $(M,w)\satisfies\possible\varphi(p_0,\ldots,p_k)$, as desired. Conversely, if
$(M,w)\satisfies\possible\varphi(p_0,\ldots,p_k)$, then there is a $u$ with $w\leq_F u$ and $(M,u)\satisfies\varphi(p_0,\ldots,p_k)$. Thus,
by induction, any $\Gamma$ forcing extension with $\Phi_u$ will satisfy $\varphi(\psi_{p_0},\ldots,\psi_{p_k})$. Since $\Phi_u$ is
forceable over any $\Gamma$ extension $W[G]$ with $\Phi_w$, it follows that any such $W[G]$ will satisfy
$\possible\varphi(\psi_{p_0},\ldots,\psi_{p_k})$, as desired.

The further final claims in the lemma now follow immediately.
\end{proof}

\section{Control statements: buttons, switches and ratchets}\label{Section.ControlStatements}

In the previous section, we reduced much of the problem of determining the modal logic of $\Gamma$ forcing to the question of whether
certain kinds of frames admit $\Gamma$-labelings. In this section, we shall prove that the existence of such labelings for large classes of
finite frames often breaks down into simpler, more modular control statements, such as what we call buttons, switches and ratchets. Since
labelings of complex frames can be constructed from these more fundamental control statements, the question of whether a given class of
frames admit labelings often reduces to the question of whether the forcing class allows for independent families of these control
statements.

Buttons and switches were introduced in \cite{HamkinsLoewe2008:TheModalLogicOfForcing}; here, we shall augment them with ratchets, weak
buttons and other types of control statements. Suppose that $\Gamma$ is a reflexive transitive forcing class. A {\df switch} for $\Gamma$
is a statement $s$ such that both $s$ and $\neg s$ are $\Gamma$ necessarily possible. A {\df button} for $\Gamma$ is a statement $b$ that
is $\Gamma$ necessarily possibly necessary. In the case that \theoryf{S4.2} is valid for $\Gamma$, this is equivalent to saying that $b$ is
possibly necessary. The button $b$ is {\df pushed} when $\necessary b$ holds, and otherwise it is {\df unpushed}. A finite collection of
buttons and switches (or other controls of this type) is {\df independent} if necessarily, each can be operated without affecting the truth
of the others. For more details, cf.\ \cite[p.\,1798]{HamkinsLoewe2008:TheModalLogicOfForcing}. A button $b$ is {\df pure} if whenever it
becomes true, it becomes necessarily true, that is, if $\necessary(b\implies\necessary b)$. Every (unpushed) button $b$ has a corresponding
(unpushed) pure button $\necessary b$, and pure buttons are sometimes more convenient.

A sequence of first-order statements $r_1$, $r_2,\ldots r_n$ is a {\df ratchet} for $\Gamma$ of length $n$ if each is an unpushed pure
button for $\Gamma$, each necessarily implies the previous, and each can be pushed without pushing the next (this notion was called a \emph{volume control} in \cite{HamkinsLoewe2008:TheModalLogicOfForcing}). This is expressed formally as follows:
$$\begin{array}{rl}
 &\neg r_i\hskip 3in\\
 &\necessary(r_i\implies\necessary r_i)\\
 &\necessary(r_{i+1}\implies r_i)\\
 &\necessary[\neg r_{i+1}\implies\possible( r_i\wedge\neg r_{i+1})]
\end{array}$$
The key idea of a ratchet is that it is unidirectional, any further $\Gamma$ forcing can only increase the ratchet value or leave it the same.
It is sometimes convenient to introduce the ratchet statement $r_0$ as any tautological statement $\top$ (an already-pushed button). A
model has ratchet value (or volume) at least $i$ when $r_i$ holds and exactly $i$ when $r_i\wedge\neg r_{i+1}$ holds, for $i<n$; ratchet
value exactly $n$ means $r_n$ is true. Further $\Gamma$ forcing pushes the value only higher, and any higher value is precisely attainable
in an appropriate extension. A ratchet of length $n$ partitions the $\Gamma$ forcing extensions into $n+1$ equivalence classes---those
having the same ratchet value---and from any model in a lower ratchet class, one can perform further $\Gamma$ forcing to arrive at a model
in any desired higher class. If $r_1$, $r_2,\ldots r_n$ is a ratchet and $s_0$, $s_1,\ldots s_k$ are switches, then this combined family is
independent if in any extension, any finite pattern of the switches is obtainable without increasing the ratchet value. A transfinite
sequence of set-theoretic statements $\<r_\alpha\st 0<\alpha<\delta>$, perhaps involving parameters, is a ratchet for $\Gamma$ of length
$\delta$ if each is an unpushed pure button for $\Gamma$, each necessarily implies the previous, and each can be pushed without pushing the
next. The ratchet is {\df uniform} if there is a formula $r(x)$ with one free variable, such that $r_\alpha=r(\alpha)$. Every finite length
ratchet is uniform. The ratchet is {\df continuous}, if for every limit ordinal $\lambda<\delta$, the statement $r_\lambda$ is equivalent
to $\forall\alpha{<}\lambda\, r_\alpha$. Any uniform ratchet can be made continuous by reindexing, replacing each $r_\beta$ by the
assertion ``$\forall\alpha{<}\beta\, r_{\alpha+1}$.'' A {\df long ratchet} is a uniform ratchet $\<r_\alpha\st 0<\alpha<\ORD>$ of length
$\ORD$, with the additional property that no $\Gamma$ forcing extension satisfies all $r_\alpha$, so that every $\Gamma$ extension exhibits
some ordinal ratchet value. We will explain in theorem \ref{Theorem.LongRatchetImpliesS4.3} that from a long ratchet, one may construct a
mutually independent family of switches and a ratchet of any desired length.

A {\df weak button} is a statement $b$ that is possibly necessary. We have mentioned that under \theoryf{S4.2}, every weak button is a
button, but without \theoryf{S4.2}, this conclusion does not follow, and it can be that a statement $b$ and its negation $\neg b$ are both
weak buttons for a given class of forcing. For example, if $\Gamma$ is c.c.c.~forcing, then the assertion ``the $\LL$-least Souslin tree
has a branch'' and its negation are both weak buttons, since one can either add a branch, which pushes the button, or specialize the tree,
which prevents branches and therefore pushes the negation. A sequence of weak buttons $b_0,\ldots,b_{n-1}$ is {\df strongly independent} if
no extension has all of them pushed, but in any extension, any additional one of them can be pushed without pushing any of the others, as long as
this wouldn't push them all. Such a family is similar to an independent family of buttons, except that one cannot push them all.

The next three theorems were implicit in \cite{HamkinsLoewe2008:TheModalLogicOfForcing}, but we revisit them here explicitly in the context
of an arbitrary forcing class $\Gamma$. Let us begin with the easiest case.

\begin{theorem}\label{Theorem.SwitchesImpliesS5}
If\/ $\Gamma$ is a reflexive transitive forcing class having arbitrarily large finite independent families of switches over a model of set
theory $W$, then the valid principles of $\Gamma$ forcing over $W$ are contained within the modal theory \theoryf{S5}.
\end{theorem}

\begin{proof}
Suppose that $\Gamma$ is a reflexive transitive forcing class and that $W$ is a model of set theory having arbitrarily large finite
independent families of switches. We may assume that the switches are all off in $W$. By theorem
\ref{Theorem.EquivalenceRelationsCompleteForS5}, any modal assertion not in \theoryf{S5} fails in a Kripke model $M$ built on a frame $F$
consisting of a single cluster of worlds $w_0$, $w_1,\ldots, w_{n-1}$, each accessible from all of them. By adding dummy copies of worlds
in this Kripke model (which does not affect modal truth), we may assume that $n=2^m$ for some natural number $m$. Let $s_0$, $s_1,\ldots
s_{m-1}$ be an independent family of $m$ switches over $W$. We will provide a $\Gamma$-labeling of this frame. For any $j<2^m$, let
$\Phi_{w_j}$ be the assertion that the overall pattern of truth values for the switches conforms with the $m$ binary digits of $j$. Thus,
we have associated each world $w_j$ of the frame of $M$ with a set-theoretic assertion $\Phi_{w_j}$. These assertions are mutually
exclusive, since different $j$ will lead to different incompatible patterns for the switches in $\Phi_j$, and exhaustive, since every model
must exhibit some pattern of switches. Since the switches are independent over $W$, any pattern of switches is $\Gamma$ necessarily
$\Gamma$ forceable, and so every $\Gamma$ extension $W[G]$ satisfies $\possible\Phi_{w_j}$ for all $j$. Since the switches are all off in
$W$, we have $W\satisfies\Phi_{w_0}$. Thus, we have verified the three labeling requirements, and so by lemma \ref{Lemma.Labeling}, there
is an assignment of the propositional variables $p$ to set-theoretic assertions $\psi_p$ such that
$(M,w_0)\satisfies\varphi(p_0,\ldots,p_k)$ if and only if $W\satisfies\varphi(\psi_{p_0},\ldots,\psi_{p_k})$. In particular, any statement
$\varphi$ that fails in $M$ at $w_0$ has a set-theoretic substitution instance $\varphi(\psi_{p_0},\ldots,\psi_{p_k})$ failing in $W$.
Since any statement outside \theoryf{S5} fails in such an $(M,w_0)$, it follows that the modal logic of $\Gamma$ forcing over $W$ is
contained within \theoryf{S5}, as desired.
\end{proof}

\begin{theorem}\label{Theorem.Ratchet+SwitchesImpliesS4.3}
If\/ $\Gamma$ is a reflexive transitive forcing class having arbitrarily long finite ratchets over a model of set theory $W$, mutually
independent with arbitrarily large finite families of switches, then the valid principles of $\Gamma$ forcing over $W$ are contained within
the modal theory \theoryf{S4.3}.
\end{theorem}

\begin{proof}
Suppose that $\Gamma$ is a reflexive transitive forcing class with arbitrarily long finite ratchets, mutually independent of switches over
a model of set theory $W$. By theorem \ref{Theorem.PreLinearOrdersCompleteForS4.3}, any modal assertion not in \theoryf{S4.3} must fail in
a Kripke model $M$ built on a finite pre-linear order frame. Thus, by lemma \ref{Lemma.Labeling}, it suffices to provide a $\Gamma$
labeling of the frame of $M$. This frame consists of a finite increasing sequence of $n$ clusters of mutually accessible worlds. That is,
the $k^{\rm th}$ cluster consists of $n_k$ many worlds $w_0^k,w_1^k,\ldots,w_{n_k-1}^k$, and the frame order is simply $w_i^k\leq w_j^s$ if
and only if $k\leq s$. By adding dummy copies of worlds in each cluster, which does not affect truth in the Kripke model, we may assume
that all clusters have the same size and furthermore, that $n_k=2^m$ for some fixed natural number $m$.

Let $r_1,\ldots,r_n$ be a ratchet of length $n$ for $\Gamma$ over $W$, mutually independent from the $m$ many switches
$s_0,\ldots,s_{m-1}$. We may assume that all switches are off in $W$. Let $\bar r_k$ be the assertion that the ratchet value is exactly
$k$, so that $\bar r_0=\neg r_1$, $\bar r_k=r_k\wedge\neg r_{k+1}$ for $1\leq k<n$ and $\bar r_n=r_n$, and let $\bar s_j$ assert for
$j<2^m$ that the pattern of switches accords with the $m$ binary digits of $j$. We associate each world $w_j^k$, where $k<n$ and $j<2^m$,
with the assertion $\Phi_{w_j^k}=\bar r_k\wedge \bar s_j$, which asserts that the ratchet value is exactly $k$ and the switches exhibit
pattern $j$. Since the ratchet values cannot go down, any pattern of switches is possible without increasing the ratchet value, and any
value is possible, it follows that the $\Gamma$ possibility of $\Phi_w$ corresponds exactly with the order in the frame. That is, whenever
$W[G]$ satisfies $\Phi_{w_j^k}$, then it satisfies $\possible_\Gamma\Phi_{w_i^s}$ if and only if $w_j^k\leq w_i^s$, which is to say, if and
only if $k\leq s$. Also, since the ratchet value is $0$ in $W$ itself and the switches are off, we have that $W\satisfies\Phi_{w_0^0}$, and
so we have provided a $\Gamma$-labeling of the frame of $M$. It follows by lemma \ref{Lemma.Labeling} that there is an assignment of the
propositional variables $p$ to set-theoretic assertions $\psi_p$ such that for any modal assertion $\varphi$ we have
$(M,w_0)\satisfies\varphi(p_0,\ldots,p_t)$ if and only if $W\satisfies\varphi(\psi_{p_0},\ldots,\psi_{p_t})$. In particular, by theorem
\ref{Theorem.PreLinearOrdersCompleteForS4.3}, any statement outside \theoryf{S4.3} will fail in such an $M$, and consequently will have a
failing substitution instance in $W$ under the $\Gamma$ forcing interpretation. Thus, the valid principles of $\Gamma$ forcing over $W$ are
contained within \theoryf{S4.3}, as desired.
\end{proof}

It turns out that in most of the set-theoretic situations where we are able to build ratchets, we are also able to build a long ratchet,
and in this case the next theorem allows us to simplify things by avoiding the need to consider switches.

\begin{theorem}\label{Theorem.LongRatchetImpliesS4.3}
If\/ $\Gamma$ is a reflexive transitive forcing class having a long ratchet over a model of set theory $W$, then the valid principles of
$\Gamma$ forcing over $W$ are contained within the modal theory \theoryf{S4.3}.
\end{theorem}

\begin{proof}
Suppose that $\<r_\alpha\st 0<\alpha<\ORD>$ is a long ratchet over $W$, that is, a uniform ratchet control of length $\ORD$, such that no
$\Gamma$ extension satisfies every $r_\alpha$. We may assume the ratchet is continuous. It suffices by theorem
\ref{Theorem.Ratchet+SwitchesImpliesS4.3} to produce arbitrarily long finite ratchets independent from arbitrarily large finite families of
switches. To do this, we shall divide the ordinals into blocks of length $\omega$, and think of the position within one such a block as
determining a switch pattern and the choice of block itself as another ratchet. Specifically, every ordinal can be uniquely expressed in
the form $\omega\cdot\alpha+k$, where $k<\omega$, and we think of this ordinal as being the $k$th element in the $\alpha$th block. Let
$s_i$ be the statement that if the current ratchet value is exactly $\omega\cdot\alpha+k$, then the $i$th binary bit of $k$ is $1$. Let
$v_\alpha$ be the assertion $r_{\omega\cdot\alpha}$, which expresses that the current ratchet value is in the $\alpha$th block of
ordinals of length $\omega$ or higher. Since we may freely increase the ratchet value to any higher value, we may increase the value of $k$
while staying in the same block of ordinals, and so the $v_\alpha$ form themselves a ratchet, mutually independent of the switches $s_i$.
Thus, by theorem \ref{Theorem.Ratchet+SwitchesImpliesS4.3}, the valid principles of $\Gamma$ forcing over $W$ are contained within
\theoryf{S4.3}.
\end{proof}

\begin{theorem}\label{Theorem.ButtonsSwitchesImpliesS4.2}
If\/ $\Gamma$ is a reflexive transitive forcing class and there are arbitrarily large finite families of mutually independent buttons and
switches over a model of set theory $W$, then the valid principles of $\Gamma$ forcing over $W$ are contained within \theoryf{S4.2}.
\end{theorem}

\begin{proof}
This was the main application of this technique in \cite{HamkinsLoewe2008:TheModalLogicOfForcing}, applied to the class of all forcing. The
point is that when you have mutually independent buttons and switches, you can label any finite pre-Boolean algebra.

Suppose that $\Gamma$ is a reflexive transitive forcing class having arbitrarily large finite families of mutually independent buttons and
switches over a model of set theory $W$. Suppose that $M$ is a Kripke model whose frame $F$ is a finite pre-Boolean algebra. Thus, the
quotient of $F$ by the equivalence relation $w\equiv v\iff w\leq v\leq w$ is a finite Boolean algebra $B$, which must be a power set $P(A)$
of a finite set $A$. Each element $a\in B$ is associated with a cluster of worlds $w^a_0,\ldots,w^a_{k_a}$. By adding dummy worlds to each
cluster, we may assume that all the clusters have size $k_a=2^m$ for some fixed $m$. Suppose that $A$ has size $n$, so that there are $n$
atoms in the Boolean algebra. Thus, the frame $F$ can be thought of as worlds $w^a_j$, where $a\of A$ and $j<2^m$, with the order
$w^a_j\leq w^c_i$ if and only if $a\of c$.

Associate each element $i\in A$ with a pure button $b_i$, such that these form a mutually independent family with $m$ many switches
$s_0,\ldots,s_{m-1}$. For $j<2^m$, let $\bar s_j$ be the assertion that the pattern of switches corresponds to the binary digits of $j$. We
label the world $w^a_j$ with the assertion $\Phi_{w^a_j}=(\bigwedge_{i\in a}b_i)\wedge\bar s_j$. If $W[G]$ satisfies $\Phi_{w^a_j}$, then
by the mutual independence of the buttons and switches, we conclude that $W[G]$ satisfies $\possible\Phi_{w^c_r}$ if and only if $a\of c$,
since any prescribed larger collection of buttons can be pushed and the switches can be set to any pattern without pushing any additional
buttons. Also, since none of the buttons is pushed and all the switches are off in $W$, we have $W\satisfies\Phi_{w^\varnothing_0}$. Thus, we
have provided a $\Gamma$ labeling of the frame for $W$.

By lemma \ref{Lemma.Labeling}, therefore, there is an assignment of the propositional variables $p\mapsto\psi_p$ such that
$(M,w^\varnothing_0)\satisfies\varphi(p_0,\ldots,p_k)$ if and only if $W\satisfies\varphi(\psi_{p_0},\ldots,\psi_{p_k})$. In particular, any
$\varphi$ failing at $(M,w^\varnothing_0)$ will have a substitutions instance failing in $W$. By theorem
\ref{Theorem.CanopiedTreesCompleteForS4.2}, any modal assertion outside \theoryf{S4.2} fails in such a Kripke model $(M,w^\varnothing_0)$,
and so the valid principles of $\Gamma$ forcing over $W$ will be contained in \theoryf{S4.2}, as desired.
\end{proof}

\begin{corollary}\label{Corollary.ExactlyS4.2}
If\/ $\Gamma$ is a reflexive transitive directed forcing class and there are arbitrarily large finite families of mutually independent
buttons and switches over a model of set theory $W$, then the valid principles of $\Gamma$ forcing over $W$ are exactly \theoryf{S4.2}.
\end{corollary}

\begin{proof}
The lower bound is provided by theorem \ref{Theorem.EasyValidities}, and the upper bound by theorem
\ref{Theorem.ButtonsSwitchesImpliesS4.2}.
\end{proof}

Let us say that a model $W$ of set theory admits a {\df uniform family of $\ORD$ many independent buttons} for $\Gamma$ forcing, if there
is a formula $\varphi$ in the language of set theory such that the assertions $b_\alpha=\varphi(\alpha)$, for each ordinal $\alpha$ in $W$,
form an independent family of buttons for $\Gamma$ forcing, and furthermore any forcing extension of $W$ pushes at most boundedly many of
the buttons. With such a large collection of buttons, as in theorem \ref{Theorem.LongRatchetImpliesS4.3}, we may avoid the need in theorem
\ref{Theorem.ButtonsSwitchesImpliesS4.2} to also have independent switches.

\begin{theorem}\label{Theorem.IndependentOrdButtonsImpliesS4.2}
If\/ $\Gamma$ is a reflexive transitive forcing class and $W$ is a model of set theory having a uniform family of $\ORD$ many independent
buttons, then the valid principles of $\Gamma$ forcing over $W$ are contained within \theoryf{S4.2}.
\end{theorem}

\begin{proof}
The idea is simply to keep the first $\omega$ many buttons as buttons and to use the rest of the buttons to form a long ratchet by looking
at the supremum of the pushed buttons beyond $\omega$. This ratchet gives rise to independent switches as in theorem
\ref{Theorem.LongRatchetImpliesS4.3}, and so one has a family of independent buttons and switches for $\Gamma$ forcing over $W$. It follows
by theorem \ref{Theorem.ButtonsSwitchesImpliesS4.2} that the valid principles of $\Gamma$ forcing over $W$ is contained in \theoryf{S4.2}.
\end{proof}

\begin{theorem}\label{Theorem.WeakButtonsImpliestBA}
If\/ $\Gamma$ is a reflexive transitive forcing class and there are arbitrarily large finite families of strongly independent weak buttons
and switches over a model of set theory $W$, then the valid principles of $\Gamma$ forcing over $W$ are contained within \theoryf{S4.tBA}.
\end{theorem}

\begin{proof}
This argument proceeds almost identically to the proof of theorem \ref{Theorem.ButtonsSwitchesImpliesS4.2}, except that here the
pre-Boolean algebras will be topless. The impossibility of pushing all the weak buttons corresponds exactly to the absence of the top
element in the Boolean algebra, so the corresponding labeling works here for the topless pre-Boolean algebra.
\end{proof}

Let us conclude this section with an aside, briefly correcting a flaw in our paper \cite{HamkinsLoewe2008:TheModalLogicOfForcing} concerning the existence of independent buttons for forcing over $\LL$. These families exist, as explained in the following, but Jakob Rittberg noticed that the particular buttons we had proposed in the proof of \cite[lemma\,6.1]{HamkinsLoewe2008:TheModalLogicOfForcing} were problematic. Specifically, we had claimed there that the buttons $b_n$ stating that ``$\omega_n^\mathrm{L}$ is not a cardinal'' form a
family of independent buttons for forcing over $\LL$. Although these are indeed buttons and one can push them in any finite pattern by
forcing over $\LL$, for true independence one must be able to continue to control the buttons independently also in all the forcing
extensions of $\LL$. And while any two of the buttons can be controlled independently in this way, what Rittberg had noticed was that if
$\LL[G]$ is a generic extension of $\LL$ in which $2^\omega$ is at least $\aleph_3$ and cardinals have not been collapsed, then the usual
collapse forcing of $\aleph_2^\LL$ to $\aleph_1^\LL$ will also collapse $\aleph_3^\LL$. So it isn't clear how to push $b_2$ over such a
model while pushing neither $b_1$ nor $b_3$. It appears to be a subtle question in forcing whether it is always possible to do so. So we do
not actually know whether the original buttons of \cite[lemma\,6.1]{HamkinsLoewe2008:TheModalLogicOfForcing} are independent over $\LL$ or
not.

Once this problem came to light, a variety of other families of independent buttons were provided. Indeed, we had already in the original article provided an alternative correct family of independent buttons and switches just after
\cite[theorem\,29]{HamkinsLoewe2008:TheModalLogicOfForcing}, and we shall presently explain these again more fully below. Rittberg himself
provided an independent family of buttons in his Master's thesis \cite[\S\,2.4.2]{Rittberg2010:TheModalLogicOfForcing} with a full proof of
their independence over $\LL$:
$$b_n^\mathrm{R} := \mbox{at least one of $\aleph_{3n}^\LL$, $\aleph_{3n+1}^\LL$, and $\aleph_{3n+2}^\LL$ is not a cardinal, or }|\wp(\aleph_{3n}^\LL)|>|\aleph_{3n+1}^\LL|.$$
Friedman, Fuchino and Sakai \cite[\S\,5]{FriedmanFuchinoSakai:OnTheSetGenericMultiverse} have another family of independent buttons,
namely, if $T_n^\LL$ denotes the $\LL$-least $\aleph_n$-Suslin tree, then any finite subfamily of the assertions
$$b_n^{\mathrm{FSS},1} := \mbox{$\aleph_n^\LL$ is not a cardinal or $T_n^\LL$ is not an $\aleph_n$-Suslin tree}$$
is an independent family of buttons over $\LL$. Furthermore, they show that the simpler statements
$$b_n^{\mathrm{FSS},2} := \mbox{there is an injection from $\aleph_{n+2}^\LL$ to $\wp(\aleph_n^\LL)$}$$
form an infinite family of independent buttons over $\LL$. For the
sake of completeness, let us give a full proof here of the independence of the alternative buttons from our original paper:
 $$b_n^*:=\mbox{$S_n$ is no longer stationary,}$$
where $\omega_1^\LL=\bigsqcup_{n\in\omega} S_n$ is the $\LL$-least partition of $\omega_1^\LL$ into $\omega$ many disjoint stationary sets
(cf.\ \cite[theorem\,29]{HamkinsLoewe2008:TheModalLogicOfForcing}). The independence of these buttons is based on our ability to control whether
a set or its complement remains stationary in a forcing extension.

\begin{lemma}[{Baumgartner, Harrington, Kleinberg \cite[thm\,23.8, ex\,23.6]{Jech:SetTheory3rdEdition}}]\label{Lemma.ShootingClubs}
If $S\of\omega_1$ is stationary, then the club-shooting forcing for $S$, consisting of the closed bounded subsets of $S$ ordered by
end-extension, adds a club subset of $S$, is countably distributive, and preserves all stationary subsets of $S$.
\end{lemma}

\begin{proof}
Let $\Q_S$ be the club-shooting forcing for $S$. Since $S$ is unbounded, it is dense that the conditions become unbounded in $S$, and so
the union of the generic filter in $\Q_S$ is a closed unbounded subset of $S$. Suppose that $T\of S$ is stationary in the ground model, and
suppose $c_0\forces\tau$ is a club subset of $\check\omega_1$. Since $T$ is stationary, we may find a countable elementary substructure
$X\elesub H_\theta$ for some large $\theta$, with $c_0,\tau,S,T\in X$ and $\delta=X\intersect\omega_1\in T$. Since $X$ is countable, we may
build a descending sequence $c_0> c_1> \cdots$ of conditions in $\Q_S\intersect X$ that get inside every open dense set in $X$. The limit
condition $c=\Union_n c_n\union\singleton{\delta}$ is a condition in $\Q_S$ precisely because $\delta\in T\of S$. Since $c$ is $X$-generic,
it follows that $c$ decides the values of any name in $X$ for a countable sequence of ordinals, and consequently the forcing is countably
distributive. Similarly, $c$ forces that $\tau$ is unbounded in $\delta$, and consequently $c$ forces that $\delta\in\tau$ and hence forces
that $\tau$ meets $\check T$. Thus, every stationary subset of $S$ is preserved to the extension.
\end{proof}

The independence of the buttons $b_n^*$ for forcing over $\LL$ now follows as an easy consequence. Namely, each $b_n^*$ is a pure button,
since non-stationarity is upward absolute, and we may in any case push all the buttons by collapsing $\omega_1$. More delicately, we may
push just $b_n^*$ by forcing as in the lemma to shoot a club through the complement of $S_n$, that is, with the club-shooting forcing for
$\bigsqcup_{m\neq n}S_m$. Lemma \ref{Lemma.ShootingClubs} exactly ensures that this forcing preserves the stationarity of any remaining
stationary $S_m$ for $m\neq n$, and thus ensures that we may push $b_n^*$ while not inadvertently pushing any other unpushed $b_m^*$.
Meanwhile, the statements $s_k$ asserting ``$2^{\aleph_k}=\aleph_{k+1}$'' are switches that can be controlled independently for $k>1$ by
countably closed forcing, which does not affect the stationarity of any subset of $\omega_1$ and therefore does not interfere with the
buttons $b_n^*$. So we have the desired independent family of buttons and switches for forcing over $\LL$.

Higher analogues of these buttons exist on higher cardinals, where one has a stationary subset $S\of\Cof_\kappa\intersect\kappa^+$ for
$\kappa$ regular and seeks to add by forcing a club set $C$ such that $C\intersect\Cof_\kappa\of S$. The natural forcing to accomplish this
is $\ltkappa$-closed and $\leq\kappa$-distributive, and the analogue of lemma \ref{Lemma.ShootingClubs} goes through.

\section{Applications to various specific classes of forcing}

In this section, we apply the results of \S\S\,\ref{Section.ForcingWithGamma} and \ref{Section.ControlStatements} and determine the valid principles of forcing for various specific natural classes of forcing.

\subsection{Our classes of forcing notions}

For any ordinal $\theta$, the {\df collapse} poset $\Coll(\omega,\theta)$ consists of the finite partial functions from $\omega$ to $\theta$, ordered by inclusion. For any nonzero ordinal $\theta$, forcing with this poset adds a function from $\omega$ onto $\theta$, making it countable in the forcing extension.
Similarly, for any ordinal $\theta$, let $\Add(\omega,\theta)$ be the set of finite partial functions from $\theta\times\omega$ into $2$ and adds $\theta$ many Cohen reals. Let $\Coll$ be the class of all forcing notions $\Q$ that are forcing equivalent to $\Coll(\omega,\theta)$ for some ordinal $\theta$ and
$\Add$ be the class of all forcing notions $\Q$ that are forcing equivalent to $\Add(\omega,\theta)$ for some ordinal $\theta$. Note that both $\Coll$ and $\Add$ include trivial forcing, since $\Add(\omega,0) = \Coll(\omega,0)=\{\varnothing\}$ is trivial (and also $\Coll(\omega,1)$ is forcing equivalent to trivial forcing). Also note that $\Add(\omega,1)$ is isomorphic to $\Coll(\omega,2)$.

Let us say that a forcing notion $\Q$ {\df necessarily collapses} $\theta$ to $\omega$ if every forcing extension by $\Q$ has a function from $\omega$ onto $\theta$ that is not in the ground model. A forcing notion $\Q$ {\df absorbs} a forcing notion $\R$, if $\Q$ is forcing equivalent to $\R*\dot\bbS$ for some (quotient) forcing $\dot\bbS$. This is equivalent to saying that $\R$ is forcing equivalent to a complete subalgebra of the Boolean algebra of $\Q$.

\begin{lemma}[Folklore] Suppose that $\theta$ is any
infinite ordinal.\label{Lemma.CollapseProperties}
\begin{enumerate}
 \item Up to forcing equivalence, $\Coll(\omega,\theta)$ is the unique forcing notion of size $|\theta|$ necessarily collapsing
     $\theta$ to $\omega$.
 \item $\Coll(\omega,\theta)$ absorbs every forcing notion of size $|\theta|$.
 \item  $\Coll(\omega,\theta)*\Coll(\omega,\lambda)$ is forcing equivalent to $\Coll(\omega,\max\{\theta,\lambda\})$.
\end{enumerate}
\end{lemma}

\begin{proof} (1). Suppose that $\Q$ is a forcing notion
of size $|\theta|$ that necessarily collapses $\theta$ to $\omega$. We may assume without loss of generality that $\Q$ is separative, since
the separative quotient of $\Q$ is forcing equivalent to it and no larger in size. Below every condition in $\Q$, we claim that there is an
antichain of size $\theta$. If $\theta$ is countable, this is immediate since every nontrivial forcing notion has infinite antichains; if
$\theta$ is uncountable, then any failure of this claim would mean that $\Q$ is $\theta$-c.c.~below some condition and consequently unable
to collapse $\theta$ to $\omega$ below that condition, contrary to our assumption. Since forcing with $\Q$ adds a function from $\omega$
onto $\theta$ and $\Q$ has size $\theta$, there is a name $\dot g$ forced to be a function from $\omega$ onto the generic filter $\dot G$.
We build a refining sequence of maximal antichains $A_n\of\Q$ as follows. Begin with $A_0=\{\onex\}$. If $A_n$ is defined, then let
$A_{n+1}$ be a maximal antichain of conditions such that every condition in $A_n$ splits into $\theta$ many elements of $A_{n+1}$, and such
that every element of $A_{n+1}$ decides the value $\dot g(\check n)$. The union $\R=\Union_n A_n$ is clearly isomorphic as a subposet of
$\Q$ to the tree $\theta^\ltomega$, and so it is forcing equivalent to $\Coll(\omega,\theta)$. Furthermore, we claim that $\R$ is dense in
$\Q$. To see this, fix any condition $q\in\Q$. Since $q$ forces that $q$ is in $\dot G$, there is some $p\leq q$ and natural number $n$
such that $p\forces_\Q\dot g(\check n)=\check q$. Since $A_{n+1}$ is a maximal antichain, there is some condition $r\in A_{n+1}$ that is
compatible with $p$. Since $r$ also decides the value of $\dot g(\check n)$ and is compatible with $p$, it must be that $r\forces \dot
g(\check n)=\check q$ also. In particular, $r$ forces $\check q\in\dot G$, and so by separativity it must be that $r\leq q$. So $\R$ is
dense in $\Q$, as desired. Thus, $\Q$ is forcing equivalent to $\R$, which we have said is forcing equivalent to $\Coll(\omega,\theta)$, as
desired.

(2). Suppose that $\Q$ has size $\theta$. Since $\Q\cross\Coll(\omega,\theta)$ also has size $\theta$ and necessarily collapses $\theta$ to
$\omega$, it follows that it is forcing equivalent to $\Coll(\omega,\theta)$. Thus, $\Q$ is absorbed by $\Coll(\omega,\theta)$. We point
out also that the quotient forcing is $\Coll(\omega,\theta)$, a fact of which we shall later make use.

(3). For any other ordinal $\lambda$, the poset $\Coll(\omega,\theta)*\Coll(\omega,\lambda)$ has size $\max\{\theta,\lambda\}$ and
necessarily collapses this ordinal to $\omega$.
\end{proof}

\begin{lemma}\label{Lemma.CollIsLinear}
The collapse forcing class $\Coll$ is a linear forcing class, which is also persistent and closed under products. \end{lemma}

\begin{proof} We have already pointed out that $\Coll$ includes trivial forcing, and so $\Coll$ is reflexive. It is transitive by lemma \ref{Lemma.CollapseProperties} (3). The same fact shows that it is linear, since the larger of $\Coll(\omega,\lambda)$ and $\Coll(\omega,\theta)$ factors through the smaller, and the quotient is collapse forcing. Finally, we observe that because the conditions are finite, the poset $\Coll(\omega,\theta)$ is absolute to all models of set theory having the ordinal $\theta$, and so $\Coll$ is persistent. It is closed under products, since $\Coll(\omega,\theta)\times\Coll(\omega,\lambda)$ has size $\max\singleton{\theta,\lambda}$ and necessarily collapses $\max\singleton{\theta,\lambda}$ to $\omega$, so by lemma \ref{Lemma.CollapseProperties} it is forcing equivalent to $\Coll(\omega,\max\singleton{\theta,\lambda})$.
\end{proof}

\begin{lemma}\label{Lemma.AddIsLinear}
The class of Cohen forcing $\Add$ is a linear forcing class, which is also persistent and closed under products.
\end{lemma}

\begin{proof}
This class is reflexive since $\Add(\omega,0)$ is trivial and transitive since $\Add(\omega,\theta)*\Add(\omega,\lambda)$ is forcing equivalent to $\Add(\omega,\theta+\lambda)$. It has the linearity property because if $\theta<\lambda$, then $\Add(\omega,\theta)*\Add(\omega,\lambda)$ is forcing equivalent to $\Add(\omega,\lambda)$. It is persistent since the definition of $\Add(\omega,\theta)$ is absolute to all models having $\theta$ as an ordinal. It is closed under products since $\Add(\omega,\theta)\times\Add(\omega,\lambda)\iso \Add(\omega,\theta+\lambda)$.
\end{proof}

The L\'evy collapse poset $\Coll(\omega,\ltlambda)$ is defined to be the finite support product $\prod_{\alpha<\lambda}\Coll(\omega,\alpha)$ and collapses all ordinals below $\lambda$ to $\omega$. This has particularly nice features when $\lambda$ is an inaccessible cardinal, but we may consider it for any ordinal $\lambda$. The class of {\df L\'evy collapse forcing}, denoted $\Coll^<$, is the class of all forcing notions that are forcing equivalent to $\Coll(\omega,\ltlambda)$ for some ordinal $\lambda$.

Lemma \ref{Lemma.CollapseProperties} (3) implies that $\Coll(\omega,\theta)$ is forcing equivalent to $\Coll(\omega,{\smalllt}(\theta+1))$ for any ordinal $\theta$, and so $\Coll\of\Coll^<$. Note that $\Coll^<$ includes the forcing to add $\omega_1$ many Cohen reals $\Add(\omega,\omega_1)\iso\Coll(\omega,{\smalllt}\omega_1)$, as well as the L\'evy collapse $\Coll(\omega,\ltkappa)$ of any inaccessible cardinal $\kappa$, if there are such cardinals.

\begin{lemma}\label{Lemma.Coll^<Properties} Let $\theta$ be an ordinal.
\begin{enumerate}
 \item If $\theta$ is not a cardinal or is singular, then $\Coll(\omega,{<}\theta)$ is forcing equivalent to $\Coll(\omega,\theta)$.
 \item $\Coll(\omega,\lttheta)*\Coll(\omega,\ltlambda)$ is forcing equivalent to $\Coll(\omega,{<}\max\{\theta,\lambda\})$.
\end{enumerate}
\end{lemma}

\begin{proof}
(1). If $\theta$ is not a cardinal or is singular, then $\Coll(\omega,{<}\theta)$ necessarily collapses $\theta$, and since this poset has
size $|\theta|$, the conclusion follows from lemma \ref{Lemma.CollapseProperties}.

(2). This follows from lemma \ref{Lemma.CollapseProperties} by combining the posets at the common stages.
\end{proof}

In general, if $\kappa$ is a cardinal, we can consider collapse-to-$\kappa$ forcing class $\Coll_\kappa$, which consists of all forcing of the form $\Coll(\kappa,\theta)$, where $\theta$ is any ordinal. Since we want to interpret the class in any model of set theory, we shall consider $\Coll_\kappa$ only when $\kappa$ is a definable regular cardinal, whose definition is absolute to all $\Coll_\kappa$ extensions. For example, we shall consider $\Coll_{\omega_1}$ and $\Coll_{\omega_2}$, and so on.

\begin{lemma}\label{Lemma.Coll_kappaLinear}
For any absolutely definable regular cardinal $\kappa$, the class $\Coll_\kappa$ is a linear forcing class, which is also persistent and closed under products.
\end{lemma}

\begin{proof}
The class $\Coll_\kappa$ is reflexive, since $\Coll(\kappa,0)=\singleton{\varnothing}$ is trivial. The class is transitive, since $\Coll(\kappa,\lambda)*\Coll(\kappa,\theta)$ is forcing equivalent to $\Coll(\kappa,\max\set{\lambda,\theta})$ by an argument analogous to lemma \ref{Lemma.CollapseProperties}. The same fact shows that $\Coll_\kappa$ is linear. Since the definition of $\kappa$ is absolute and posets in $\Coll_\kappa$ do not add new sets of size less than $\kappa$, it follows that $\Coll_\kappa$ is persistent. From this and transitivity, it follows that $\Coll_\kappa$ is closed under products.
\end{proof}

In the same spirit, we define the class of all $\kappa$-Cohen forcing, denoted by $\Add_\kappa$, consisting of
all forcing notions of the form $\Add(\kappa,\theta)$, the poset to add $\theta$ many Cohen subsets to $\kappa$, having conditions that are partial functions from $\theta\times\kappa$ to $2$ of size less than $\kappa$. We consider this class only when $\kappa$ is a definable regular cardinal, whose definition is absolute to all $\Add_\kappa$ extensions. In this case, the class $\Add_\kappa$ is a linear forcing class by essentially the same argument as in lemma \ref{Lemma.AddIsLinear}.

A forcing notion $\P$ has {\df essential size} $\delta$, if the complete Boolean algebra corresponding to $\P$ has size $\delta$ and $\P$ is not forcing equivalent to any $\Q$ whose complete Boolean algebra has smaller size. If $W$ is any model of set theory, we can define the {\df forcing distance from $\LL$} as the least $\LL$-cardinal $\delta$ such that $W$ can be written as $\LL[G]$ where $G$ is $\P$-generic over $\LL$ for a forcing notion $\P$ of essential size $\delta$. We write $\mathrm{fd}_\LL = \delta$ in this case.

\begin{lemma}\label{lem:esssizeratchet}
If $W\models\mathrm{fd}_\LL = \delta$ and $H$ is $\Q$-generic over $W$ for some $\Q\in W$, then $W[H]\models\mathrm{fd}_\LL \geq \delta$.
\end{lemma}

\begin{proof}
Let $W[H] = \LL[G]$ for some $G$ which is $\P$-generic over $\LL$ for some $\P\in\LL$ and let $\B$ be the complete Boolean algebra corresponding to $\P$. Since $\LL\subseteq W \subseteq W[H] = \LL[G]$, we know (by \cite[lemma 15.43]{Jech:SetTheory3rdEdition}) that there is a subalgebra $\C$ of $\B$ such that $W = \LL[G']$ for some $G'$ which is $\C$-generic over $\LL$. We assumed that $\mathrm{fd}^W_\LL = \delta$, and thus $|\C|^\LL \geq \delta$. Since $\C$ was a subalgebra of $\B$, we have that $|\B|^\LL \geq \delta$.
\end{proof}

\subsection{The modal logic of collapse forcing}\label{Section.CollapseForcing}

We aim to determine the provably valid principles of collapse forcing.
Lemma \ref{Lemma.CollIsLinear} and theorem \ref{Theorem.EasyValidities} give us \theoryf{S4.3} as a lower bound.

\begin{theorem}\label{Theorem.CollValid=S4.3}
If\/ \ZFC\ is consistent, then the \ZFC-provably valid principles of collapse forcing $\Coll$ are exactly those in \theoryf{S4.3}.
\end{theorem}

\begin{proof}
For the upper bound, we shall show that $\Coll$ admits a long ratchet over the constructible universe $\LL$. For each non-zero ordinal $\alpha$, let $r_\alpha$ be the statement ``$\aleph_\alpha^\LL$ is countable.'' These statements form a long ratchet for collapse forcing over the constructible universe $\LL$, since any collapse extension $\LL[G]$ collapses an initial segment of the cardinals of $\LL$ to $\omega$, and in any such extension in which $\aleph_\alpha^\LL$ is not yet collapsed, the forcing to collapse it will not yet collapse $\aleph_{\alpha+1}^\LL$. Thus, by theorem \ref{Theorem.LongRatchetImpliesS4.3}, the valid principles of collapse forcing over $\LL$ are contained within \theoryf{S4.3}. So the valid principles of collapse forcing over $\LL$ are precisely \theoryf{S4.3}, and if\/ \ZFC\ is consistent, then the \ZFC-provably valid principles of collapse forcing are exactly \theoryf{S4.3}.
\end{proof}

\begin{theorem}\label{Theorem.ContainingColl}
If\/ $\Gamma$ is a transitive reflexive forcing class and necessarily $\Coll\of\Gamma$, then the valid principles of $\Gamma$ forcing over $\LL$ contain \theoryf{S4.2} and are contained within \theoryf{S4.3}.
\end{theorem}

\begin{proof} Suppose that $\Gamma$ is a transitive reflexive forcing class necessarily containing collapse forcing. By lemma \ref{Lemma.CollapseProperties}, it follows that every poset is absorbed by any sufficiently large collapse poset, and so $\Gamma$ is directed. Thus, by theorem \ref{Theorem.EasyValidities}, the modal theory \theoryf{S4.2} is valid for $\Gamma$ forcing over any model of set theory.

Let us now establish the upper bound of \theoryf{S4.3} by finding a long ratchet over the constructible universe $\LL$. For each nonzero ordinal $\alpha$, let $r_\alpha$ be the assertion ``Some $\LL$ cardinal above $\aleph_\alpha^\LL$ is collapsed.'' Clearly, these statements are all false in $\LL$, each implies its own necessity, each implies the previous, and if $r_{\alpha+1}$ is not yet true in a forcing extension $\LL[G]$, then we may collapse $\aleph_{\alpha+1}^\LL$ without collapsing any larger cardinal, making $r_\alpha\wedge\neg r_{\alpha+1}$ true. So these statements form a long ratchet over $\LL$, and therefore, by theorem \ref{Theorem.LongRatchetImpliesS4.3}, the
valid principles of $\Gamma$ forcing over $\LL$ are included within \theoryf{S4.3}. \end{proof}

\begin{corollary}\label{Corollary.ContainingCollAndLinear=S4.3}
If\/ \ZFC\ is consistent and $\Gamma$ is a linear forcing class containing collapse forcing $\Coll$, then the \ZFC-provably valid principles of $\Gamma$ forcing are exactly \theoryf{S4.3}.
\end{corollary}

\begin{proof}
If $\Gamma$ is a linear forcing class, then \theoryf{S4.3} is valid by theorem \ref{Theorem.EasyValidities}. And the validities are
contained with \theoryf{S4.3} by the previous theorem.
\end{proof}

\begin{theorem}
If\/ \ZFC\ is consistent, then the \ZFC-provably valid principles of L\'evy collapse forcing $\Coll^<$ is exactly \theoryf{S4.3}.
\end{theorem}

\begin{proof}
It follows by lemma \ref{Lemma.Coll^<Properties} that $\Coll^<$ is a linear forcing class containing $\Coll$, and so this theorem is an instance of corollary \ref{Corollary.ContainingCollAndLinear=S4.3}.
\end{proof}

\begin{theorem}\label{Theorem.Coll_kappaValid=S4.3}
If\/ \ZFC\ is consistent, then the \ZFC-provably valid principles of collapse-to-$\kappa$ forcing, for any absolutely definable regular cardinal $\kappa$, are exactly \theoryf{S4.3}.
\end{theorem}

\begin{proof}
By lemma \ref{Lemma.Coll_kappaLinear}, the collapse-to-$\kappa$ forcing class $\Coll_\kappa$ is a linear forcing class, and so by theorem \ref{Theorem.EasyValidities}, every \theoryf{S4.3} assertion is valid for collapse-to-$\kappa$ forcing. Conversely, consider $\kappa$ as it is defined in $\LL$, namely, $\kappa^\LL=\aleph_\beta^\LL$ for some ordinal $\beta$. For each nonzero ordinal $\alpha$, let $r_\alpha$ be
the assertion ``some $\LL$-cardinal above $\aleph_{\beta+\alpha}^\LL$ is collapsed.'' Just as in the previous section, these form a long ratchet for collapse-to-$\kappa$ forcing over $\LL$, because in any $\Coll_\kappa$ forcing extension $\LL[G]$, if no cardinals above $\aleph_{\beta+\alpha}^\LL$ are collapsed, then we can collapse $\aleph_{\beta+\alpha}^\LL$ to $\kappa$ without collapsing any larger $\LL$-cardinals. By theorem \ref{Theorem.LongRatchetImpliesS4.3}, it follows that the valid principles of collapse-to-$\kappa$ forcing over $\LL$ are contained within \theoryf{S4.3}. Thus, if \ZFC\ is consistent, it follows that the \ZFC-provably valid principles of collapse-to-$\kappa$ forcing are exactly \theoryf{S4.3}.
\end{proof}

\begin{theorem}\label{Theorem.ArbSizeS4.3}
Suppose that $\Gamma$ is a transitive reflexive forcing class, with the property that there is a definable proper class $C$ of cardinals in $\LL$, such that in any forcing extension $\LL[G]$ by a forcing notion in $\Gamma$ having essential size $\delta_0$ and any larger $\delta\in C$, there is a forcing notion in $\Gamma^{\LL[G]}$ having essential size $\delta$. Then the valid principles of $\Gamma$ forcing over $\LL$ are contained within \theoryf{S4.3}.
\end{theorem}

\begin{proof}
We shall construct a long ratchet. For each nonzero ordinal $\alpha$, let $w_\alpha$ be the assertion ``$\mathrm{fd}_\LL$ is bigger than the $\alpha$th element of $C$''. By lemma \ref{lem:esssizeratchet}, each statement $w_\alpha$ is an unpushed pure button in $\LL$. Our assumptions on $\Gamma$ and $C$ ensure that if a set forcing extension $\LL[G]$ does not yet satisfy $w_\alpha$, then $w_\alpha\wedge\neg w_{\alpha+1}$ is $\Gamma$ forceable.  Finally, since any set forcing extension $\LL[G]$ was obtained by forcing of {\it some} size, no such extension can satisfy all $w_\alpha$. Thus, this is indeed a long ratchet, and so the valid principles of $\Gamma$ forcing over $\LL$ are contained within \theoryf{S4.3} by theorem \ref{Theorem.LongRatchetImpliesS4.3}.
\end{proof}

\begin{theorem}\label{Theorem.ContainingCollkappa}
If\/ $\Gamma$ is any reflexive transitive forcing class necessarily containing $\Coll_\kappa$ for some $\Gamma$-absolutely definable regular cardinal $\kappa$, then the valid principles for $\Gamma$ forcing over $\LL$ are contained within \theoryf{S4.3}.
\end{theorem}

\begin{proof} Suppose that $\Gamma$ and $\kappa$ are as stated. Check that the conditions of
theorem \ref{Theorem.ArbSizeS4.3} are satisfied: The fact that $\kappa$ is $\Gamma$-absolutely definable guarantees that $C := \{(\kappa^{\plus^{(\alpha)}})^\LL:\alpha\in\mathrm{Ord}\}$ is a definable proper class; if $\LL[G]$ is a $\Gamma$ extension of essential size $\delta_0 < \delta \in C$, then
$\Coll(\kappa,\delta)\in\Gamma$ has essential size $\delta$. Now the claim follows from theorem \ref{Theorem.ArbSizeS4.3}.
\end{proof}

\subsection{The modal logic of Cohen forcing}\label{Section.CohenForcing}

\begin{theorem}\label{Theorem.AddValid=S4.3}
If\/ \ZFC\ is consistent, then the \ZFC-provably valid principles of Cohen forcing $\Add$ are exactly \theoryf{S4.3}.
\end{theorem}

\begin{proof}
By lemma \ref{Lemma.AddIsLinear}, the class $\Add$ is a linear forcing class, and so the valid principles of Cohen forcing include
\theoryf{S4.3} by theorem \ref{Theorem.EasyValidities}. For the upper bound, we shall construct a long ratchet for Cohen forcing over
$\LL$. For each nonzero ordinal $\alpha$, let $r_\alpha$ be the statement that $2^\omega>\aleph_\alpha^\LL$. By adding additional subsets
to $\omega$, we can push up the value of the continuum to any such prescribed degree. And in any Cohen extension of $\LL$, all cardinals
have been preserved, so the continuum can be pushed up so as to realize any particular $\aleph_{\beta+1}^\LL$ above the current size of the
continuum, thereby forcing $r_\beta\wedge\neg r_{\beta+1}$. Thus, this is indeed a long ratchet for Cohen forcing $\Add$ over $\LL$, and so
by theorem \ref{Theorem.LongRatchetImpliesS4.3}, the valid principles of Cohen forcing over $\LL$ are contained within \theoryf{S4.3}, and
consequently are exactly equal to \theoryf{S4.3}. In particular, if\/ \ZFC\ is consistent, then the \ZFC-provably valid principles of Cohen
forcing are exactly \theoryf{S4.3}.
\end{proof}

As mentioned above, the class $\Add_\kappa$ is a linear forcing class, and therefore \theoryf{S4.3} is always valid for $\Add_\kappa$ forcing.

\begin{theorem}\label{Theorem.Add_kappaValid=S4.3}
If $\kappa$ is an absolutely definable regular cardinal, then the valid principles of $\kappa$-Cohen forcing $\Add_\kappa$ over $\LL$ are exactly \theoryf{S4.3}. Consequently, if\/ \ZFC\ is consistent, then the \ZFC-provably valid principles of $\kappa$-Cohen forcing are exactly \theoryf{S4.3}.
\end{theorem}

\begin{proof}
This argument proceeds as in theorem \ref{Theorem.Coll_kappaValid=S4.3}, but using the analogue of the long ratchet of theorem \ref{Theorem.AddValid=S4.3}. As mentioned, \theoryf{S4.3} is always valid for $\Add_\kappa$ forcing. Conversely, consider $\kappa$ as defined in $\LL$, which must be $\aleph_\beta^\LL$ for some ordinal $\beta$. For each nonzero ordinal $\alpha$, let $r_\alpha$ be the statement that $2^\kappa>\aleph_{\beta+\alpha}^\LL$. This is a long ratchet over $\LL$ with respect to $\Add_\kappa$, since we can always push up $2^\kappa$ to attain any particular higher cardinal value, and since $\Add_\kappa$ forcing does not collapse cardinals over $\Add_\kappa$ extensions $\LL[G]$, once $r_\alpha$ is true in an $\Add_\kappa$ extension, it remains true in all further extensions. Thus, by theorem \ref{Theorem.LongRatchetImpliesS4.3}, the valid principles of $\Add_\kappa$ forcing over $\LL$ are contained within \theoryf{S4.3}, and therefore are in fact equal to \theoryf{S4.3}. It follows that if \ZFC\ is consistent, then the \ZFC-provably valid principles of $\Add_\kappa$ forcing are precisely those in \theoryf{S4.3}.
\end{proof}

\subsection{Upper bounds for other classes of forcing}\label{Section.OtherClasses}

We shall now apply theorem \ref{Theorem.ArbSizeS4.3} to other forcing classes:

\begin{corollary}\label{Corollary.VariousClassesInS4.3}
The valid principles of forcing over $\LL$ for each of the following forcing classes is contained within \theoryf{S4.3}. For none of the
forcing classes do the validities over $\LL$ include \theoryf{S4.2}.
\begin{enumerate}
 \item c.c.c.~forcing.
 \item Proper forcing.
 \item Semi-proper forcing.
 \item Stationary-preserving forcing.
 \item $\omega_1$-preserving forcing.
 \item Cardinal-preserving forcing.
 \item Countably distributive forcing.
\end{enumerate}
\end{corollary}

\begin{proof}
Each of these classes is easily seen to be transitive and reflexive. Let $C$ be the class of uncountable successor cardinals of $\LL$. For any $\delta\in C$, each class contains forcing of essential size $\delta$. For example, $\Add(\omega,\delta)$ is in classes (1) through (6), and $\Add(\delta,1)$ is in class (7). The same is true in any forcing extension $\LL[G]$ by set forcing of essential size below $\delta$. Thus, each class satisfies the requirements of theorem \ref{Theorem.ArbSizeS4.3}, and so the valid principles of forcing in each
case are included within \theoryf{S4.3}.

We complete the proof by showing for each of the classes that $\axiomf{.2}$ is not valid over $\LL$. Let $\varphi$ be the assertion ``the $\LL$-least Souslin tree is a special Aronszajn tree.'' This statement is c.c.c.~forceable over $\LL$, simply by specializing the tree, and once the statement is true, it remains true in all $\omega_1$-preserving extensions. Consider now the c.c.c.~extension $\LL[b]$ obtained by forcing to add a branch $b$ through $T$. In this extension and all subsequent $\omega_1$-preserving extensions, $\varphi$ is false. Since each of the forcing classes in (1) through (6) contains c.c.c.~forcing and is contained within $\omega_1$-preserving forcing, this statement $\varphi$ is a violation of axiom $\axiomf{.2}$ with respect to each of the forcing classes. For class (7), we may consider the forcing to kill the stationary of the $\LL$-least stationary co-stationary subset of $\omega_1$, or to kill its complement, and the impossibility of doing both while adding no reals provides a violation of $\axiomf{.2}$.
\end{proof}

The reader can probably extend this list with additional natural forcing classes satisfying the hypothesis of theorem
\ref{Theorem.ArbSizeS4.3} (the classes of countably closed forcing and more generally, $\kappa$-closed forcing for any fixed absolutely definable cardinal $\kappa$, will be covered in \S\,\ref{ssec:countablyclosed}). Let us state for the record that we do not currently know the exact forcing validities for any of the classes listed in corollary \ref{Corollary.VariousClassesInS4.3}. Nevertheless, our analysis of the case of $\omega_1$-preserving forcing and cardinal-preserving will be improved in theorem \ref{Theorem.Omega1PrservingS4tBA}.

\subsection{Countably closed and
$\kappa$-closed forcing}\label{ssec:countablyclosed}

We turn now to the class of countably closed forcing, a highly natural forcing class with robust closure properties, along with its generalization to $\ltkappa$-closed forcing. As usual, a forcing notion $\Q$ is {\df countably closed} if every decreasing sequence $q_0\geq q_1\geq\cdots$ has a lower bound in $\Q$. This property is not preserved by equivalence of forcing (for example, an atomless complete Boolean algebra is {\it never} countably complete, since one may take the supremum of a countable antichain, stripping off one element in each step), so we understand the class of countably closed forcing to refer to the class of all forcing notions that are forcing equivalent to a countably closed forcing notion.

\begin{theorem}\label{Theorem.CountablyClosed=S4.2}
If\/ \ZFC\ is consistent, then the \ZFC-provably valid principles of the class of countably closed forcing are exactly those in \theoryf{S4.2}.
\end{theorem}

\begin{proof}
Since any countably closed forcing notion remains countably closed in any countably closed extension, it follows that the class of countably closed forcing is persistent. Thus, by theorem \ref{Theorem.EasyValidities}, the validities always include \theoryf{S4.2}. For the upper bound, consider countably closed forcing over $\LL$. Consider the higher buttons and switches introduced in \S\,\ref{Section.ForcingWithGamma}, either the higher analogues we mentioned of our own buttons $b_n^*$ or the higher buttons of Rittberg, or of Friedman, Fuchino and Sakai, which can in each case be controlled independently via countably closed forcing. Thus, by theorem \ref{Theorem.ButtonsSwitchesImpliesS4.2}, the valid principles of countably closed forcing over $\LL$ are contained within \theoryf{S4.2}. In summary, if \ZFC\ is consistent, then the \ZFC-provably valid principles of countably closed forcing are precisely the assertions of \theoryf{S4.2}.
\end{proof}

A forcing notion $\Q$ is $\ltkappa$-closed if every descending sequence in $\Q$ of length less than $\kappa$ has a lower bound in $\Q$. Thus, the countably closed posets are exactly the ${\smalllt}\aleph_1$-closed. A cardinal $\kappa$ is {\it absolutely defined} by a formula $\varphi$ in the context of a reflexive transitive forcing class $\Gamma$, if $\kappa$ is the only object satisfying $\varphi(x)$ in any $\Gamma$-forcing extension.

\begin{theorem}\label{Theorem.KappaClosed=S4.2}
More generally, if\/ \ZFC\ is consistent, then for any absolutely definable regular cardinal $\kappa$, the \ZFC-provably valid principles of $\kappa$-closed forcing are exactly those in \theoryf{S4.2}.
\end{theorem}

\begin{proof}
Since the class is persistent, we again obtain \theoryf{S4.2} as a lower bound. And for the upper bound, one may use the higher analogues of the buttons and switches used in theorem \ref{Theorem.CountablyClosed=S4.2}, by translating them above $\kappa$.
\end{proof}

\subsection{$\omega_1$-preserving forcing}

Surely the class of cardinal-preserving forcing and the larger class of $\omega_1$-preserving forcing have been focal points of forcing theory. We have already observed in corollary \ref{Corollary.VariousClassesInS4.3} that the valid principles of $\omega_1$-preserving forcing over $\LL$ is contained within \theoryf{S4.3}. Here, we prove a stronger conclusion, although the exact modal theory of validities for this class remains an open question.

\begin{theorem}\label{Theorem.Omega1PrservingS4tBA}
The valid principles of $\omega_1$-preserving forcing over $\LL$ are included within $\theoryf{S4.tBA}\subsetneqq\theoryf{S4.2}$.
\end{theorem}

\begin{proof}
By theorem \ref{Theorem.WeakButtonsImpliestBA}, it suffices to find arbitrarily large finite strongly independent families of weak buttons and switches. For this, we may use the same buttons $b_n^*$ mentioned in the context of lemma \ref{Lemma.ShootingClubs}, which are described also in \cite[thm\,29]{HamkinsLoewe2008:TheModalLogicOfForcing}. These were a independent family of buttons for the class of all forcing, but for the class of $\omega_1$-preserving forcing, they become a strongly independent family of weak buttons, since in any $\omega_1$-preserving extension $\LL[G]$, at least one of the sets $S_n$ must still remain stationary, and so not all buttons are pushed. But subject to this requirement, lemma \ref{Lemma.ShootingClubs} and the subsequent argument shows that the buttons can be controlled
independently by $\omega_1$-preserving forcing. For a family of switches that is mutually independent of these weak buttons, consider $(2^\omega)^{\LL[G]}$, which must be $\aleph_{\omega\cdot\xi+k}^\LL$ for some ordinal $\xi$ and some natural number $k$, and let $s_j$ assert that the $j$th binary digit of $k$ is $1$. Since we can force any desired value for $k$ by adding Cohen reals, it follows that any finite pattern of the $s_j$ switches is realizable, by forcing that preserves all stationary sets, and hence does not inadvertently push any of the unpushed weak buttons $b_n$. Since we have a strongly independent family of weak buttons, with a mutually independent family of switches, it follows by theorem \ref{Theorem.WeakButtonsImpliestBA} that the valid principles of $\omega_1$-preserving forcing over $\LL$ is contained within \theoryf{S4.tBA}.
\end{proof}

We are unsure whether this argument can be extended to the class of cardinal-preserving forcing, since it seems that in some extensions, the club-shooting forcing could conceivably collapse cardinals above $\omega_1$. Nevertheless, this issue may be solvable by the alternative method of shooting clubs, via finite conditions, which does preserve all cardinals.

\subsection{$\CH$-preserving forcing}\label{ssec:preserving}

In any model of $\ZFC+\varphi$, we say that forcing notion $\Q$ is {\df $\varphi$-preserving}, if every forcing extension by $\Q$ satisfies $\varphi$. Note that we consider only $\varphi$-preserving forcing in models that actually satisfy $\varphi$, so by ``provably valid principles of $\varphi$-preserving forcing'', we mean the $\ZFC+\varphi$-provably valid principles of $\varphi$-preserving forcing.

\begin{theorem}\label{thm:chpreserving}
If\/ \ZFC\ is consistent, then the provably valid principles of \CH-preserving forcing, of \GCH-preserving forcing and of $\neg\CH$-preserving forcing are all exactly \theoryf{S4.2}.
\end{theorem}

\begin{proof} Let us treat the case of \CH\ first. The class of \CH-preserving forcing, over any model satisfying \CH, is easily seen to be reflexive and transitive. So \theoryf{S4} is valid. Let us argue that .2 is also valid for $\CH$-preserving forcing.

Suppose that $V\satisfies\CH$ and $V\satisfies\possible\necessary\varphi$ as well.  Assume that $V$ does not satisfy .2, that is, that $V\satisfies\neg\necessary\possible\varphi$, or, equivalently, $V\satisfies\possible\necessary\neg\varphi$.  Then there are \CH-preserving extensions $V[G]$ and $V[H]$ where $V[G]\satisfies\necessary\varphi$ and $V[H]\satisfies\necessary\neg\varphi$.  We may assume without loss of generality that $G$ and $H$ are mutually $V$-generic. We don't know that $V[G*H]$ satisfies \CH, so we can't just directly combine them. But let $V[G*H][I]$ be a further forcing extension of $V[G*H]$ that does satisfy \CH.
We claim that $V[G*H][I]$ is the result of $\CH$-preserving forcing over $V[G]$.  Consider some partial order $\P$ such that there is some $\P$-generic $J$ such that $V[G*H][I]=V[G][J]$.  Since this model satisfies \CH, there is some $p\in J$ such that $p$ forces \CH.  Let $\P_p$ be the poset $\P$ below $p$, i.e., $\P_p$ is $\CH$-preserving.  Without loss of generality, we can assume that $J$ is $\P_p$-generic, thus yielding the claim.
Thus, $V[G*H][I]$ is obtained from $V[G]$ by \CH-preserving forcing, satisfying both $\varphi$ and $\neg\varphi$, a contradiction.
Thus, $V$ satisfies \axiomf{.2} as desired, and by observation
\ref{obs:nec.2}, $V$ satisfied $\necessary\axiomf{.2}$. Since the class of $\CH$-preserving forcings is closed under iterations, the Kripke model corresponding to every forcing translation is transitive. Hence, by observation \ref{obs:necessitation}, this means that \theoryf{S4.2} is valid in $V$ for \CH-preserving forcing, thus (since $V$ was arbitrary) giving the lower bound.

For the upper bound, consider \CH-preserving forcing over $\LL$. We may use the independent family of buttons and switches for $\LL$ presented in \S\,\ref{Section.ControlStatements}, which work just as well for \CH-preserving forcing, and so by theorem \ref{Theorem.ButtonsSwitchesImpliesS4.2} the valid principles of \CH-preserving forcing over $\LL$ are contained within \theoryf{S4.2}. So they are exactly \theoryf{S4.2}.

Essentially the same argument works in the case of $\neg\CH$-preserving forcing, by means of the higher analogues of the buttons and switches, which can be controlled with highly closed and therefore $\neg\CH$-preserving forcing. For \GCH-preserving forcing, however, we shall need to use different switches. To find them, consider any forcing extension $\LL[G]$, and let $\beta$ be the least ordinal such that there are no $\LL$-generic Cohen subsets in $\LL[G]$ of any cardinal above $\aleph_{\beta}^\LL$. The ordinal $\beta$ can be written uniquely in the form $\omega\cdot\alpha+k$ for some ordinal $\alpha$ and $k<\omega$. Let $s_m$ be the statement that the $m$th binary
digit of $k$ is $1$. These form an independent family of switches (and independent from the buttons $b_n$ above), because by adding a Cohen set up high, we can realize any desired natural number $k$, and hence realize any desired finite pattern in the switches. (Alternatively, we could build a long ratchet with a similar method, mutually independent from the buttons $b_n$, and this also suffices.)
\end{proof}

\subsection{The modal logic of c.c.c.~forcing}\label{Section.CCCforcing}

The class of c.c.c.~forcing is the pre-eminent forcing class, noticed already in the earliest days of forcing as a central notion. Since the class of c.c.c.~partial orders is not closed under the equivalence of forcing (for example, every poset is forcing equivalent to very large lottery sums of the poset with itself), let us understand the class of c.c.c.~forcing to consist of all forcing notions that are forcing equivalent to a c.c.c.~partial order. In corollary \ref{Corollary.VariousClassesInS4.3}, we observed that the valid principles of c.c.c.~forcing over $\LL$ are contained within \theoryf{S4.3} and do not contain \theoryf{S4.2}. But we do not know the exact modal theory of c.c.c.~forcing validities over $\LL$ or the class of \ZFC-provably valid principles of c.c.c.~forcing.

\begin{theorem}\label{Theorem.MAimpliesS4.2forCCC}
If\/ $\MA_{\omega_1}$ holds, then \theoryf{S4.2} is valid for c.c.c.~forcing.
\end{theorem}

\begin{proof}
Since the class of c.c.c.~forcing is clearly reflexive and transitive, it follows by theorem \ref{Theorem.EasyValidities} at least that \theoryf{S4} is valid. Martin's axiom $\MA_{\omega_1}$ implies that c.c.c.~forcing in $V$ is persistent, that is, every c.c.c.~poset in $V$ remains c.c.c.~in every c.c.c.~extension of $V$ \cite[theorem 16.21]{Jech:SetTheory3rdEdition}. This exactly shows that axiom $\axiomf{.2}$ is valid for c.c.c.~forcing in $V$. As in the proof of theorem \ref{thm:chpreserving}, observations \ref{obs:nec.2} and \ref{obs:necessitation} give us that \theoryf{S4.2} is valid for c.c.c.~forcing in every model of $\MA_{\omega_1}$.
\end{proof}

It should be pointed out that the (modal logic) proof of theorem \ref{Theorem.MAimpliesS4.2forCCC} hides the set-theoretic fact that the $\MA_{\omega_1}$ hypothesis is rather fragile, and easily destroyed by c.c.c.~forcing (e.g., adding a Cohen real creates Souslin trees and therefore destroys \MA). Our theorem tells us that \theoryf{S4.2} will continue to be valid in all further c.c.c.~forcing extensions of a model of $\MA_{\omega_1}$, even though such extensions may no longer have $\MA_{\omega_1}$. (In particular, it means that while Souslin trees are created, they cannot be definable in a way that allows us to create the counterexamples to $\axiomf{.2}$, as in corollary \ref{Corollary.VariousClassesInS4.3}).

The argument generalizes beyond c.c.c.~forcing to any transitive reflexive forcing class $\Gamma$. Namely, if such a $\Gamma$ is persistent in some model $V$, then it will continue to satisfy \theoryf{S4.2} in all $\Gamma$ extensions $V[G]$, even if the interpretation of $\Gamma$ in those later models is no longer persistent.

Finally, we remark that in order to prove that the
$\ZFC+\MA_{\omega_1}$-provably valid principles of c.c.c.~forcing are exactly \theoryf{S4.2}, it would be sufficient
to identify arbitrarily large finite families of independent c.c.c.-buttons and switches over the Solovay-Tennenbaum model $\LL[G]$ of $\MA+\neg\CH$.

\bibliographystyle{abbrv}
\bibliography{s43}
\end{document}